\documentclass{article}

\usepackage{ssArxiv}

\usepackage[utf8]{inputenc}

\usepackage[title]{appendix}

\usepackage{amsmath,amsbsy,amsgen,amscd,amssymb,amsthm,amsfonts}
\usepackage{mathtools}
\usepackage{bm}
\usepackage{enumitem}

\usepackage{microtype}
\usepackage{hyperref}
\usepackage[nameinlink]{cleveref}
\usepackage{url} 
\usepackage{graphicx}
\usepackage{algorithm,algorithmic}
\usepackage{natbib}
\usepackage{xspace}

\newtheorem{lemma}{Lemma}
\newtheorem{corollary}{Corollary}
\newtheorem{remark}{Remark}
\newtheorem{theorem}{Theorem}

\crefname{assumption}{assumption}{assumptions}

\newcommand{\norm}[1]{\|#1\|}
\newcommand{\dom}{\mathrm{dom}}
\newcommand{\prox}{\mathrm{prox}}
\newcommand{\proj}{\mathrm{proj}}
\newcommand{\dist}{\mathrm{dist}}
\newcommand{\ip}[2]{\langle#1,#2\rangle}
\newcommand{\R}{\mathbb{R}}

\newcommand{\algo}{\textsc{AdapTos}\xspace}

\newcommand{\Prob}{\hyperref[eqn:model-problem]{Problem~(1)}\xspace}
\newcommand{\Probsto}{\hyperref[eqn:model-problem-stochastic]{Problem~\eqref{eqn:model-problem-stochastic}}\xspace}
\newcommand{\ProbAdapTos}{\hyperref[eqn:model-problem-adaptos]{Problem~\eqref{eqn:model-problem-adaptos}}\xspace}

\makeatletter
\def\blfootnote{\gdef\@thefnmark{}\@footnotetext}
\makeatother
  
\title{Three Operator Splitting with Subgradients, \\Stochastic Gradients, and Adaptive Learning Rates}
\author{%
\name Alp Yurtsever$^\star$ \email{alp.yurtsever@umu.se}\\
\addr{Umeå University, Umeå, Sweden} \\[0.5em]
\name Alex Gu$^\star$ \email{gua@mit.edu}\\
\name Suvrit Sra \email{suvrit@mit.edu}\\
\addr{Massachusetts Institute of Technology, Cambridge, MA, USA} 
}

\begin{document}

\maketitle

\blfootnote{Alp Yurtsever and Alex Gu contributed equally to this paper. The paper is based primarily on the work done while Alp Yurtsever was at Massachusetts Institute of Technology.}

\begin{abstract}
Three Operator Splitting (TOS) \citep{davis2017three} can minimize the sum of multiple convex functions effectively when an efficient gradient oracle or proximal operator is available for each term. This requirement often fails in machine learning applications: \emph{(i)}~instead of full gradients only stochastic gradients may be available; and  \emph{(ii)}~instead of proximal operators, using subgradients to handle complex penalty functions may be more efficient and realistic. Motivated by these concerns, we analyze three potentially valuable extensions of TOS. The first two permit using subgradients and stochastic gradients, and are shown to ensure a $\smash{\mathcal{O}(1/\sqrt{t})}$ convergence rate. The third extension \algo endows TOS with adaptive step-sizes. For the important setting of optimizing a convex loss over the intersection of convex sets \algo attains universal convergence rates, \textit{i.e.}, the rate adapts to the \emph{unknown} smoothness degree of the objective function. We compare our proposed methods with competing methods on various applications.  
\end{abstract}

\section{Introduction}
We study convex optimization problems of the form
\begin{equation}
\label{eqn:model-problem}
\min_{x \in \R^n} \quad \phi(x) := f(x) + g(x) + h(x),
\end{equation}
where $\smash{f: \mathbb{R}^n \to \mathbb{R}}$ and $\smash{g, h: \mathbb{R}^{n} \to \mathbb{R} \cup \{+\infty\}}$ are proper, lower semicontinuous and convex functions.  Importantly, this template captures constrained problems via indicator functions. To avoid pathological examples, we assume that the relative interiors of $\smash{\dom(f)}$, $\smash{\dom(g)}$ and $\smash{\dom(h)}$ have a nonempty intersection. 

\Prob is motivated by a number of applications in machine learning, statistics, and signal processing, where the three functions comprising the objective $\phi$ model data fitting, structural priors, or decision constraints. Examples include overlapping group lasso~\citep{yuan2011efficient}, isotonic regression~\citep{tibshirani2011nearly}, dispersive sparsity~\citep{el2015totally}, graph transduction \citep{shivanna2015spectral}, learning with correlation matrices~\citep{higham2016anderson}, and multidimensional total variation denoising~\citep{barbero2018modular}.

An important technique for addressing composite problems is \emph{operator splitting}~\citep{bauschke2011convex}. 
However, the basic proximal-(sub)gradient method may be unsuitable for \Prob since it requires the prox-operator of $g+h$, computing which may be vastly more expensive than individual prox-operators of $g$ and $h$. An elegant, recent method, Three Operator Splitting (TOS, \citet{davis2017three}, see \Cref{alg:three-operator-splitting}) offers a practical choice for solving \Prob when $f$ is smooth. 
Importantly, at each iteration, TOS evaluates the gradient of $f$ and the proximal operators of $g$ and $h$ only once. 
Moreover, composite problems with more than three functions can be reformulated as an instance of \Prob in a product-space and solved by using TOS. This is an effective method as long as each function has an efficient gradient oracle or proximal operator (see \Cref{sec:related-works}). 

Unfortunately, TOS is not readily applicable to many optimization problems that arise in machine learning. Most important among those are problems where only access to stochastic gradients is feasible, \textit{e.g.}, when performing large-scale empirical risk minimization and online learning. Moreover, prox-operators for some complex penalty functions are computationally expensive and it may be more efficient to instead use  subgradients. For example, proximal operator for the maximum eigenvalue function that appears in dual-form semidefinite programs (\textit{e.g.}, see Section~6.1 in \citep{ding2019optimal}) may require computing a full eigendecomposition. In contrast, we can form a subgradient by computing only the top eigenvector via power method or Lanczos algorithm. 

\textbf{Contributions}. With the above motivation, this paper contributes three key extensions of TOS. We tackle nonsmoothness in \Cref{sec:algorithm-nonsmooth} and stochasticity in \Cref{sec:algorithm-stochastic}. These two extensions enable us to use subgradients and stochastic gradients of $f$ (see \Cref{sec:related-works} for a comparison with related work), and satisfy a  $\smash{\mathcal{O}(1/\sqrt{T})}$ error bound in function value after $T$ iterations. The third main contribution is \algo in \Cref{sec:adaptos}. This extension provides an adaptive step-size rule in the spirit of AdaGrad \citep{duchi2011adaptive,levy2017online} for an important subclass of \Prob. Notably, for optimizing a convex loss over the intersection of two convex sets, \algo ensures universal convergence rates. That is, \algo implicitly adapts to the \emph{unknown} smoothness degree of the problem, and ensures a $\smash{\tilde{\mathcal{O}}(1/\sqrt{t})}$ convergence rate when the problem is nonsmooth but the rate improves to $\smash{\tilde{\mathcal{O}}(1/t)}$ if the problem is smooth and a solution lies in the relative interior of the feasible set. 

In \Cref{sec:numerical-experiments}, we discuss empirical performance of our methods by comparing them against present established methods on various benchmark problems from COPT Library \citep{copt} including the overlapping group lasso, total variation deblurring, and sparse and low-rank matrix recovery. We also test our methods on nonconvex optimization by training a neural network model. We present more experiments on isotonic regression and portfolio optimization in the supplements. 

\textbf{Notation.} We denote a solution of \Prob by $x_\star$ and $\phi_\star := \phi(x_\star)$. The distance between a point $x \in \R^n$ and a closed and convex set $\mathcal{G} \subseteq \R^n$ is $\dist(x,\mathcal{G}) := \min_{y \in \mathcal{G}} \norm{x - y}$; the projection of $x$ onto $\mathcal{G}$ is given by $\proj_{\mathcal{G}}(x) := \arg\min_{y \in \mathcal{G}} \norm{x - y}$. 
The prox-operator of a function $g: \R^n \to \R \cup \{+\infty\}$ is defined by $\prox_{g}(x) := \arg\min_{y \in \R^n} \{ g(y) + \frac{1}{2} \norm{x - y}^2 \}$. 
The indicator function of $\mathcal{G}$ gives $0$ for all $x \in \mathcal{G}$ and $+\infty$ otherwise. 
Clearly, the prox-operator of an indicator function is the projection onto the corresponding set. 

\section{Background and related work}
\label{sec:related-works}
TOS, proposed recently by \citet{davis2017three}, can be seen as a generic extension of various operator splitting schemes, including the forward-backward splitting, Douglas-Rachford splitting, forward-Douglas-Rachford splitting~\citep{briceno2015forward}, and the generalized forward-backward splitting~\citep{raguet2013generalized}. It covers these aforementioned approaches as special instances when the terms $f,g$ and $h$ in \Prob are chosen appropriately.  Convergence of TOS is well studied when $f$~has Lipschitz continuous gradients. It ensures $\smash{\mathcal{O}(1/t)}$ convergence rate in this setting, see \citep{davis2017three} and \citep{pedregosa2016convergence} for details. 

Other related methods that can be used for \Prob when $f$ is smooth are the primal-dual hybrid gradient (PDHG) method \citep{condat2013primal,vu2013splitting} and the primal-dual three operator splitting methods in \citep{yan2018new} and \citep{salim2020dualize}. These methods can handle a more general template where $g$ or $h$ is composed with a linear map, however, they require $f$ to be smooth. 
The convergence rate of PDHG is studied in \citep{chambolle2016ergodic}.

\textit{Nonsmooth setting.} We are unaware of any prior result that permits using subgradients in TOS (or in other methods that can use the prox-operator of $g$ and $h$ separately for \Prob). The closest match is the proximal subgradient method which applies when $h$ is removed from \Prob, and it is covered by our nonsmooth TOS as a special case. 

\textit{Stochastic setting.} %
There are multiple attempts to devise a stochastic TOS in the literature. 
\citet{yurtsever2016stochastic} studied \Prob under the assumption that $f$ is smooth and strongly convex, and an unbiased gradient estimator with bounded variance is available. Their stochastic TOS has a guaranteed $\mathcal{O}(1/t)$ convergence rate. In \citep{cevher2018stochastic}, they drop the strong convexity assumption, instead they assume that the variance is summable. They show asymptotic convergence with no guarantees on the rate. 
Later, \citet{pedregosa2019proximal} proposed a stochastic variance-reduced TOS and analyzed its non-asymptotic convergence guarantees. Their method gets $\mathcal{O}(1/t)$ convergence rate when $f$ is smooth. The rate becomes linear if $f$ is smooth and strongly convex and $g$ (or $h$) is also smooth. 
Recently, \citet{yurtsever2021three} studied TOS on problems where $f$ can be nonconvex and showed that the method finds a first-order stationary point with $\smash{\mathcal{O}(1/\sqrt[3]{t})}$ convergence rate under a diminishing variance assumption. They increase the batch size over the iterations to satisfy this assumption. 

None of these prior works cover the broad template we consider: $f$~is smooth or Lipschitz continuous and the stochastic first-order oracle has bounded variance. To our knowledge, our paper gives the first analysis for stochastic TOS without strong convexity assumption or variance reduction. 

Other related methods are the stochastic PDHG in \citep{zhao2018stochastic}, the decoupling method in \citep{mishchenko2019stochastic}, the stochastic primal-dual method in \citep{zhao2019optimal}, and the stochastic primal-dual three operator splitting in \citep{salim2020dualize}. The method in \citep{zhao2019optimal} can be viewed as an extension of stochastic ADMM \citep{ouyang2013stochastic,azadi2014towards} from the sum of two terms to three terms in the objective. Similar to the existing stochastic TOS variants, these methods either assume strong convexity or require variance-reduction. 

\textit{Adaptive step-sizes.} The standard writings of TOS and PDHG require the knowledge of the smoothness constant of $f$ for the step-size. Backtracking line-search strategies (for finding a suitable step-size when the smoothness constant is unknown) are proposed for PDHG in \citep{malitsky2018first} and for TOS in \citep{pedregosa2018adaptive}. 
These line-search strategies are significantly different than our adaptive learning rate. Importantly, these methods work only when $f$ is smooth. They require extra function evaluations, and are thus not suitable for stochastic optimization. And their goal is to estimate the \emph{smoothness constant}. In contrast, our goal is to design an algorithm that adapts to the unknown \emph{smoothness degree}. Our method does not require function evaluations, and it can be used in smooth, nonsmooth, or stochastic settings. 

At the heart of our method lie adaptive online learning algorithms \citep{duchi2011adaptive,rakhlin2013optimization} together with online to offline conversion techniques \citep{levy2017online,cutkosky2019anytime}. Similar methods appear in the literature for other problem templates with no constraint or a single constraint in \citep{levy2017online,levy2018online,kavis2019unixgrad,cutkosky2019anytime,bach2019universal}. Our method extends these results to optimization over the intersection of convex sets. When $f$ is nonsmooth, \algo ensures a $\smash{\tilde{\mathcal{O}}(1/\sqrt{t})}$ rate, whereas the rate improves to $\smash{\tilde{\mathcal{O}}(1/t)}$ if $f$ is smooth and there is a solution in the relative interior of the feasible set.

\textbf{TOS for more than three functions.} 
TOS can be used for solving problems with more than three convex functions by a product-space reformulation technique \citep{briceno2015forward}. Consider  
\begin{equation} \label{eqn:problem-multi-term}
\min_{x \in \R^d} \quad \sum_{i=1}^q \phi_i(x),
\end{equation}
where each component $\phi_i:\R^d \to \R \cup \{+\infty\}$ is a proper, lower semicontinuous and convex function. Without loss of generality, suppose $\phi_1, \ldots, \phi_p$ are prox-friendly.
Then, we can reformulate \eqref{eqn:problem-multi-term} in the product-space $\smash{\R^{d \times (p+1)}}$ as 
\begin{equation}\label{eqn:problem-reformulated}
\min_{(x_0, x_1, \ldots, x_p) \in \R^{d\times(p+1)}} \quad \sum_{i = 1}^p \phi_i(x_i) + \sum_{i = p+1}^{q} \phi_{i}(x_0) \quad \text{subject to} \quad x_0 = x_1 = \ldots = x_p.
\end{equation}
This is an instance of \Prob with $n=d\times(p+1)$ and $x = (x_0,x_1,\ldots,x_p)$. We can~choose $g(x)$ as the indicator of the equality constraint, $f(x) = \sum_{i = p+1}^{q} \phi_{i}(x_0)$, and $h(x) = \sum_{i = 1}^p \phi_i(x_i)$. Then, the (sub)gradient of $f$ is the sum of (sub)gradients of $\phi_{p+1},\ldots,\phi_q$; $\prox_g$ is a mapping that averages $x_0,x_1,\ldots,x_p$; and $\prox_h$ is the concatenation of the individual $\prox$-operators of $\phi_1,\ldots,\phi_p$.

TOS has been studied only for problems with smooth $f$, and this forces us to assign all nonsmooth components $\phi_i$ in \eqref{eqn:problem-multi-term} to the proximal term $h$ in \eqref{eqn:problem-reformulated}. In this work, by enabling subgradient steps for nonsmooth $f$, we provide the flexibility to choose how to process each nonsmooth component $\phi_i$ in \eqref{eqn:problem-reformulated}, either by its proximal operator through $h$ or by its subgradient via $f$. 

\section{TOS for Nonsmooth Setting}
\label{sec:algorithm-nonsmooth}

\Cref{alg:three-operator-splitting} presents the generalized TOS for \Prob. It recovers the standard version in \citep{davis2017three} if we choose $u_t = \nabla f(z_t)$ when $f$ is smooth. For convenience, we define the mapping
\begin{align}\label{eqn:TOS-operator}
    \texttt{TOS}_\gamma(y, u) := y - \prox_{\gamma g}(y) + \prox_{\gamma h}\big(2 \cdot \prox_{\gamma g}(y) - y - \gamma u\big)
\end{align}
which represents one iteration of \Cref{alg:three-operator-splitting}.

The first step of the analysis is the fixed-point characterization of TOS. The following lemma is a straightforward extension of Lemma~2.2 in \citep{davis2017three} to permit subgradients. The proof is similar to \citep{davis2017three}, we present it in the supplementary material for completeness. %

\begin{algorithm}[tb]
   \caption{Three Operator Splitting (TOS)}
   \label{alg:three-operator-splitting}
\begin{algorithmic}
\vspace{0.25em}
   \STATE {\bfseries Input:} Initial point $y_0 \in \mathbb{R}^n$, step-size sequence $\{\gamma_t\}_{t=0}^T$ \\[0.25em]
   \FOR{$t=0,1,2,\ldots,T$}
   \STATE $z_{t} = \prox_{\gamma_t g} (y_t)$
   \STATE Choose an update direction $u_t \in \mathbb{R}^n$ \hfill  {\color{darkgreen}\COMMENT{\footnotesize $u_t = \nabla f(z_t)$ captures the standard version of TOS}}~~~
   \STATE $x_{t} = \prox_{\gamma_t h} (2z_t - y_t - \gamma_t u_t)$
   \STATE $y_{t+1} = y_t - z_t + x_t $
   \ENDFOR\\[0.25em]
  \STATE {\bfseries Return:} Ergodic sequence $\bar{x}_t$ and $\bar{z}_t$ defined in \eqref{eqn:ergodic-sequence}
   \vspace{0.25em}
\end{algorithmic}
\end{algorithm}

\begin{lemma}[Fixed points of TOS]
\label{lem:fixed-point-characterization}
Let $\gamma > 0$. Then, there exists a subgradient $\smash{u \in \partial f(\prox_{\gamma g}(y))}$ that satisfies $\smash{\emph{\texttt{TOS}}_\gamma(y,u) = y}$ if and only if $\smash{\prox_{\gamma g}(y)}$ is a solution of \Prob.
\end{lemma}

When $f$ is $L_f$-smooth, \texttt{TOS} with $u_t = \nabla f(z_t)$ is known to be an averaged operator\footnote{An operator $\texttt{T}:\mathbb{R}^n \to \mathbb{R}^n$ is $\omega$-averaged if $\norm{\texttt{T}x - \texttt{T}y}^2 \leq \norm{x-y}^2 - \frac{1-\omega}{\omega} \norm{(x-\texttt{T}x)-(y-\texttt{T}y) }^2$ for some $\omega \in (0,1)$ for all $x,y \in \mathbb{R}^n$.} %
if $\gamma \in (0,2/L_f)$ (see Proposition~2.1 in \citep{davis2017three}) and the analysis in prior work is based on this property. In particular, averagedness implies Fej\'er monotonicity, \textit{i.e.}, that $\norm{y_t - y_\star}$ is non-increasing, where $y_\star$ denotes a fixed point of \texttt{TOS}. 
However, when $f$ is nonsmooth and $u_t$ is replaced with a subgradient, \texttt{TOS} operator is no longer averaged and the standard analysis fails. One of our key observations is that $\norm{y_t - y_\star}$ remains bounded even-though we loose averagedness and Fej\'er monotonicity in this setting, see \Cref{thm:nonsmooth-y-bounded} in the supplements.

\textbf{Ergodic sequence.} 
Convergence of operator splitting methods are often given in terms of ergodic (averaged) sequences. This strategy requires maintaining the running averages of $z_t$ and $x_t$:
\begin{align}\label{eqn:ergodic-sequence}
    \bar{x}_t = \frac{1}{t+1} \sum_{\tau = 0}^t x_\tau
    \qquad \text{and} \qquad
    \bar{z}_t = \frac{1}{t+1} \sum_{\tau = 0}^t z_\tau. 
\end{align}
Clearly, we do not need to store the history of $x_t$ and $z_t$ to maintain these sequences. In practice, the last iterate often converges faster than the ergodic sequence. We can evaluate the objective function at both points and return the one with the smaller value.

We are ready to present convergence guarantees of TOS for the nonsmooth setting. 

\begin{theorem}
\label{thm:convergence-Lipschitz}
Consider \Prob and %
employ TOS (\Cref{alg:three-operator-splitting}) with the update directions and step-size chosen as
\begin{equation}
    u_t \in \partial f(z_t)~~~~ \text{and}~~~~ \gamma_t = \frac{\gamma_0}{\sqrt{T+1}} ~~\text{for some $\gamma_0 > 0$}, \quad \text{for $t = 0,1,\ldots,T$.}
\end{equation}
Assume that $\norm{u_t} \leq G_f$ for all $t$. 
Then, the following guarantees hold:
\begin{gather}
f(\bar{z}_T) + g(\bar{z}_T) + h(\bar{x}_T) - \phi_\star \leq \frac{1}{2\sqrt{T+1}} \left(\frac{D^2}{\gamma_0} + \gamma_0 G_f^2 \right) \label{eqn:thm1-objective}  \\
\text{and} \quad \norm{\bar{x}_T - \bar{z}_T} \leq  \frac{2}{T+1} \left( D + \gamma_0 G_f  \right), \quad \text{where} \quad D = \max\{\norm{y_0 - x_\star}, \norm{y_0 - y_\star}\}. \label{eqn:thm1-constraint} 
\end{gather}
\vspace{0em}
\end{theorem}

\begin{remark}
The boundedness of subgradients is a standard assumption in nonsmooth optimization. It is equivalent to assuming that $f$ is $G_f$-Lipschitz continuous on $\dom(g)$.
\end{remark}

If $D$ and $G_f$ are known, we can optimize the constants in \eqref{eqn:thm1-objective} by choosing $\gamma_0 = D/G_f$. This gives $f(\bar{z}_T) + g(\bar{z}_T) + h(\bar{x}_T) - \phi_\star \leq \mathcal{O}(D G_f / \sqrt{T})$ and $\norm{\bar{x}_T - \bar{z}_T} \leq \mathcal{O}(D/T)$. 

\begin{proof}[Proof sketch]
We start by writing the optimality conditions for the proximal steps for $z_t$ and $x_t$. Through algebraic modifications and by using convexity of $f$, $g$ and $h$, we obtain
\begin{align}
f(z_t)  + g(z_t) + h (x_t) - \phi_\star \leq \frac{1}{2\gamma} \norm{y_t - x_\star}^2 - \frac{1}{2\gamma} \norm{y_{t+1} - x_\star}^2  + \frac{\gamma}{2} \norm{u_t}^2.
\end{align}
$\norm{u_t} \leq G_f$ by assumption. Then, we average this inequality over $t = 0, 1 , \ldots, T$ and use Jensen's inequality to get \eqref{eqn:thm1-objective}. 

The bound in \eqref{eqn:thm1-constraint} is an immediate consequence of the boundedness of $\norm{y_{T+1}-y_\star}$ that we show in \Cref{thm:nonsmooth-y-bounded} in the supplementary material:
\begin{align}
\norm{y_{T+1} - y_\star} 
\leq \norm{y_0-y_\star} + 2 \gamma_0 G_f.
\end{align}
By definition, $\norm{\bar{x}_T-\bar{z}_T} = \frac{1}{T} \norm{y_{T+1} - y_0} \leq \frac{1}{T}(\norm{y_{T+1} - y_\star} + \norm{y_\star - y_0})$.
\end{proof}

\Cref{thm:convergence-Lipschitz} does not immediately yield convergence to a solution of \Prob because $f+g$ and $h$ are evaluated at different points in \eqref{eqn:thm1-objective}. 
Next corollary solves this issue.  

\begin{corollary}\label{cor:tos-nonsmooth}
We are interested in two particular cases of \Cref{thm:convergence-Lipschitz}: 

\emph{(i).} Suppose $h$ is $G_h$-Lipschitz continuous. Then, 
\begin{align}
\phi(\bar{z}_T) - \phi_\star 
& \leq  \frac{1}{2\sqrt{T+1}} \left(\frac{D^2}{\gamma_0}  + \gamma_0 G_f^2 \right) 
+ \frac{2G_h}{T+1} \left( D + \gamma_0 G_f \right).
\end{align}

\emph{(ii).} Suppose $h$ is the indicator function of a convex set $\mathcal{H} \subseteq \mathbb{R}^n$. Then, 
\begin{align}
f(\bar{z}_T) + g(\bar{z}_T) - \phi_\star & \leq \frac{1}{2\sqrt{T+1}} \left(\frac{D^2}{\gamma_0} + \gamma_0 G_f^2 \right) \\
\text{and} \quad \dist(\bar{z}_T,\mathcal{H}) & \leq \frac{2}{T+1} \left( D + \gamma_0 G_f \right).
\end{align}
\end{corollary}

\begin{proof}
(i). Since $h$ is $G_h$-Lipschitz, $\phi(\bar{z}_T) \leq f(\bar{z}_T) + g(\bar{z}_T) + h(\bar{x}_T) + G_h \norm{\bar{x}_T - \bar{z}_T}$. \\[0.2em]
(ii). $h(\bar{x}_T) = 0$ since $\bar{x}_T \in \mathcal{H}$. Moreover, $\dist(\bar{z}_T,\mathcal{H}) := \inf_{x \in \mathcal{H}} \norm{\bar{z}_T - x} \leq \norm{\bar{z}_T - \bar{x}_T}$.
\end{proof}

\begin{remark}
We fix time horizon $T$ for the ease of analysis and presentation. In practice, we use $\gamma_t = {\gamma_0}/{\sqrt{t+1}}$. %
\end{remark}

\Cref{thm:convergence-Lipschitz} covers the case in which $g$ is the indicator of a convex set $\mathcal{G} \subseteq \mathbb{R}^n$. By definition, $\bar{z}_T \in \mathcal{G}$ and $x_\star \in \mathcal{G}$, hence $g(\bar{z}_T) = g(x_\star) = 0$. If both $g$ and $h$ are indicator functions, TOS gives an approximately feasible solution, in $\mathcal{G}$, and close to $\mathcal{H}$. We can also consider a stronger notion of approximate feasibility, measured by $\dist(\bar{z}_T,\mathcal{G} \cap \mathcal{H})$. However, this requires additional regularity assumptions on $\mathcal{G}$ and $\mathcal{H}$ to avoid pathological examples, see Lemma~1 in \citep{hoffmann1992distance} and Definition~2 in \citep{kundu2018convex}.
 
\Prob captures unconstrained minimization problems when $g = h= 0$. Therefore, the convergence rate in \Cref{thm:convergence-Lipschitz} is optimal in the sense that it matches the information theoretical lower bounds for first-order black-box methods, see~Section~3.2.1 in \citep{nesterov2003introductory}.
Remark that the subgradient method can achieve a $\mathcal{O}(1/t)$ rate when $f$ is strongly convex. We leave the analysis of TOS for strongly convex nonsmooth $f$ as an open problem. 

\section{TOS for Stochastic Setting}
\label{sec:algorithm-stochastic}

In this section, we focus on the three-composite stochastic optimization template:
\begin{equation} \label{eqn:model-problem-stochastic}
\min_{x \in \R^n} \quad \phi(x) := f(x) + g(x) + h(x) \quad \text{where} \quad f(x) := \mathbb{E}_{\xi} \tilde{f}(x,\xi)
\end{equation}
and $\xi$ is a random variable. 
The following theorem characterizes the convergence rate of \Cref{alg:three-operator-splitting} for \Probsto. 

\begin{theorem}\label{thm:stochastic}
Consider \Probsto and employ TOS (\Cref{alg:three-operator-splitting}) with a fixed step-size $\gamma_t = \gamma = \gamma_0 / \sqrt{T+1}$ for some $\gamma_0 > 0$. Suppose we are receiving the update directions $u_t$ from an unbiased stochastic first-order oracle with bounded variance, i.e.,
\begin{align}
    \hat{u}_t := \mathbb{E}[u_t | z_t] \in \partial f(z_t) \quad \text{and} \quad \mathbb{E}[\norm{u_t - \hat{u}_t}^2] \leq \sigma^2  ~~ \text{for some $\sigma < +\infty$.}
\end{align}
Assume that $\norm{\hat{u}_t} \leq G_f$ for all $t$. Then, the following guarantees hold:
\begin{gather}
\mathbb{E} [f(\bar{z}_T) + g(\bar{z}_T) + h (\bar{x}_T)] - \phi_\star 
 \leq \frac{1}{2\sqrt{T+1}} \left(\frac{ D^2}{\gamma_0} + \gamma_0 (\sigma^2 + G_f^2)  \right) ~~~ \text{and} \label{eqn:thm-stochastic-subgradient-a} \\
 \mathbb{E}[\norm{\bar{x}_T - \bar{z}_T}]  \leq \frac{2}{T+1} \left( D + \gamma_0 \left(G_f + \frac{\sigma}{2}\right)\right), ~~ \text{where} ~~ D = \max\{\norm{y_0 - x_\star}, \norm{y_0 - y_\star}\}.
\label{eqn:thm-stochastic-subgradient-b}
\end{gather}
\end{theorem}

 \vspace{0.25em}
\begin{remark}\label{rem:smooth-stochastic}
Similar rate guarantees hold with some restrictions on the choice of $\gamma_0$ if we replace bounded subgradients assumption with the smoothness of $f$. We defer details to the supplements. 
\end{remark}

If we can estimate $D, G_f$ and $\sigma$, then we can optimize the bounds by choosing $\gamma_0 \approx D/\max\{G_f, \sigma\}$. This gives $\smash{f(\bar{z}_T) + g(\bar{z}_T) + h(\bar{x}_T) - \phi_\star \leq \mathcal{O}(D \max\{G_f,\sigma\} / \sqrt{T})}$ and $\smash{\norm{\bar{x}_T - \bar{z}_T} \leq \mathcal{O}(D/T)}$. 

Analogous to \Cref{cor:tos-nonsmooth}, from \Cref{thm:stochastic} we can derive convergence guarantees when $h$ is Lipschitz continuous or an indicator function. As in the nonsmooth setting, the rates shown in this section are optimal because \Probsto covers $g(x) = h(x) = 0$  as a special case. 

\section{TOS with Adaptive Learning Rates}
\label{sec:adaptos}

In this section, we focus on an important subclass of \Prob where $g$ and $h$ are indicator functions of some closed and convex sets:
\begin{equation}
\label{eqn:model-problem-adaptos}
\min_{x \in \R^n} \quad f(x) \quad \text{subject to} \quad x\in \mathcal{G} \cap \mathcal{H}. 
\end{equation}
TOS is effective for \ProbAdapTos when projections onto $\cal G$ and $\cal H$ are easy but the projection onto their intersection is challenging. Particular examples include transportation polytopes, doubly nonnegative matrices, and isotonic regression, among many others. 

We propose \algo with an adaptive step-size in the spirit of adaptive online learning algorithms and online to batch conversion techniques, see \citep{duchi2011adaptive, rakhlin2013optimization, levy2017online, levy2018online, cutkosky2019anytime, kavis2019unixgrad,bach2019universal} and the references therein. 
\algo employs the following step-size rule:
\begin{align} \label{eqn:adaptos-step-size}
    \gamma_t = \frac{\alpha}{\sqrt{\beta + \sum_{\tau=0}^{t-1} \norm{u_\tau}^2}} \quad \text{for some $\alpha, \beta > 0$.}
\end{align}
$\beta$ in the denominator prevents $\gamma_t$ to become undefined. If $D := \norm{y_0 - x_\star}$ and $G_f$ are known, theory suggests choosing $\alpha = D$ and $\beta = G_f^2$ for a tight upper bound, however, this choice affects only the constants and not the rate of convergence as we demonstrate in the rest of this section. Importantly, we do not assume any prior knowledge on $D$ or $G_f$. In practice, we often discard $\beta$ and use $\gamma_0 = \alpha$ at the first iteration. 

For \algo, in addition to \eqref{eqn:ergodic-sequence}, we will also use a second ergodic sequence with weighted averaging:
\begin{align}
        \tilde{x}_t = \frac{1}{\sum_{\tau = 0}^t \gamma_\tau} \sum_{\tau = 0}^t \gamma_\tau x_\tau
    \qquad \text{and} \qquad
    \tilde{z}_t = \frac{1}{\sum_{\tau = 0}^t \gamma_\tau} \sum_{\tau = 0}^t \gamma_\tau z_\tau. 
    \label{eqn:ergodic-sequence-2}
\end{align}
This sequence was also considered for TOS with line-search in \citep{pedregosa2018adaptive}.

\begin{theorem}\label{thm:adaptos-nonsmooth}
Consider \ProbAdapTos and TOS (\Cref{alg:three-operator-splitting}) with the update directions $u_t \in \partial f(z_t)$ and the adaptive step-size \eqref{eqn:adaptos-step-size}. Assume that $\norm{u_t} \leq G_f$ for all $t$. Then, the estimates generated by TOS satisfy 
\begin{gather}
f(\tilde{z}_t) - f_\star 
\leq \tilde{\mathcal{O}} \left( \frac{2\alpha G_f}{\sqrt{t+1}} \Big(\tfrac{D^2}{4\alpha^2} + 1 + \tfrac{G_f}{\sqrt{\beta}} \Big) \right) \quad \text{and} \\[0.5em]
\quad \dist(\bar{z}_t,\mathcal{H}) 
\leq \tilde{\mathcal{O}}\bigg(\frac{2\alpha}{\sqrt{t+1}}\Big(1 + \tfrac{G_f}{\sqrt{\beta}}\Big)\bigg) ~~~ \text{where} ~~~ D = \norm{y_0 - x_\star}. 
\end{gather}
\end{theorem}
If $D$ and $G_f$ are known, we can choose $\alpha = D$ and $\smash{\beta = G_f^2}$.
This gives $\smash{f(\tilde{z}_t) - f_\star \leq \tilde{\mathcal{O}}(G_f D / \sqrt{t})}$ and $\smash{\dist(\bar{z}_t,\mathcal{H}) \leq \tilde{\mathcal{O}}(D/\sqrt{t})}$. 

The next theorem establishes a faster rate for the same algorithm when $f$ is smooth and a solution lies in the interior of the feasible set. 

\begin{theorem}\label{thm:adaptos-smooth}
Consider \ProbAdapTos and suppose $f$ is $L_f$-smooth on $\mathcal{G}$. Use TOS (\Cref{alg:three-operator-splitting}) with the update directions $u_t = \nabla f(z_t)$ and the adaptive step-size \eqref{eqn:adaptos-step-size}. Assume that $\norm{u_t} \leq G_f$ for all $t$. Suppose \ProbAdapTos has a solution in the interior of the feasible set. Then, the estimates generated by TOS satisfy
\begin{gather}
f(\bar{z}_t) - f_\star 
 \leq \tilde{\mathcal{O}} \left( \frac{2}{t+1} \Big( 4 \alpha^2 L_f \big(\tfrac{D^2}{4\alpha^2} + 1 + \tfrac{G_f^2}{\beta} \big)^2  + \alpha \sqrt{\beta} \big(\tfrac{D^2}{4\alpha^2} + 1 + \tfrac{G_f^2}{\beta} \big) \Big)  \right) \quad \text{and} \\[0.5em]
 \dist(\bar{z}_t,\mathcal{H}) 
 \leq \tilde{\mathcal{O}} \bigg( \frac{2\alpha}{t+1} \Big(\tfrac{D}{\alpha} + 1 + \tfrac{G_f}{\sqrt{\beta}} \Big) \bigg) ~~~ \text{where} ~~~ D = \norm{y_0 - x_\star}. 
\end{gather}
\end{theorem}
If $D$ and $G_f$ are known, we can choose $\alpha = D$ and $\smash{\beta = G_f^2}$. \\
This gives $\smash{f(\bar{z}_t) - f_\star \leq \tilde{\mathcal{O}}((L_fD^2+G_fD)/t)}$ and $\smash{\dist(\bar{z}_t,\mathcal{H}) \leq \tilde{\mathcal{O}}(D/t)}$. 

\begin{remark}
When $f$ is smooth, the boundedness assumption $\norm{u_t}\leq G_f$ holds automatically with $G_f \leq L_f D_{\mathcal{G}}$ if $\mathcal{G}$ has a bounded diameter $D_{\mathcal{G}}$.
\end{remark}

We believe the assumption on the location of the solution is a limitation of the analysis and that the method can achieve fast rates when $f$ is smooth regardless of where the solution lies. Remark that this assumption also appears in  \citep{levy2017online,levy2018online}.

Following the definition in \citep{nesterov2015universal}, we say that an algorithm is universal if it does not require to know whether the objective is smooth or not yet it implicitly adapts to the smoothness of the objective. \algo attains universal convergence rates for \ProbAdapTos. 
It converges to a solution with $\smash{\tilde{\mathcal{O}}(1/\sqrt{t})}$ rate (in function value) when $f$ is nonsmooth. The rate becomes $\smash{\tilde{\mathcal{O}}(1/t)}$ if $f$ is smooth and the solution is in the interior of the feasible set. 

Finally, the next theorem shows that \algo can successfully handle stochastic (sub)gradients. 

\begin{theorem}\label{thm:adaptos-stochastic}
Consider \ProbAdapTos. Use TOS (\Cref{alg:three-operator-splitting}) with the update directions $u_t$ from an unbiased stochastic subgradient oracle such that $\mathbb{E}[u_t|z_t] \in \partial f(z_t)$ almost surely. Assume that $\norm{u_t} \leq G_f$ for all $t$. Suppose \ProbAdapTos has a solution in the interior of the feasible set. Then, the estimates generated by TOS satisfy 
\begin{gather}
\mathbb{E}\big[f(\tilde{z}_t) - f_\star\big]
\leq \tilde{\mathcal{O}} \bigg( \frac{2 \alpha G_f}{\sqrt{t+1}} \Big(\tfrac{D^2}{4\alpha^2} + 1 + \tfrac{G_f^2}{\beta}\Big) \bigg) \quad \text{and} \\[0.5em]
\quad \mathbb{E}[\dist(\bar{z}_t,\mathcal{H})] 
\leq \tilde{\mathcal{O}}\bigg(\frac{2\alpha}{t+1}\Big(\tfrac{D}{\alpha} + 1 + \tfrac{G_f}{\sqrt{\beta}}\Big)\bigg) ~~~ \text{where} ~~~ D = \norm{y_0 - x_\star}. 
\end{gather}
\end{theorem}

\section{Numerical Experiments}
\label{sec:numerical-experiments}

This section demonstrates empirical performance of the proposed method on a number of convex optimization problems. We also present an experiment on neural networks. Our experiments are performed in Python 3.7 with Intel Core i9-9820X CPU @ 3.30GHz. We present more experiments on isotonic regression and portfolio optimization in the supplementary materials. The source code for the experiments is available in the supplements. 

\subsection{Experiments on Convex Optimization with Smooth \texorpdfstring{$f$}{f}}
\label{sec:experiments-smooth}

\begin{figure}[!t]
\begin{center}
\includegraphics[width=\linewidth]{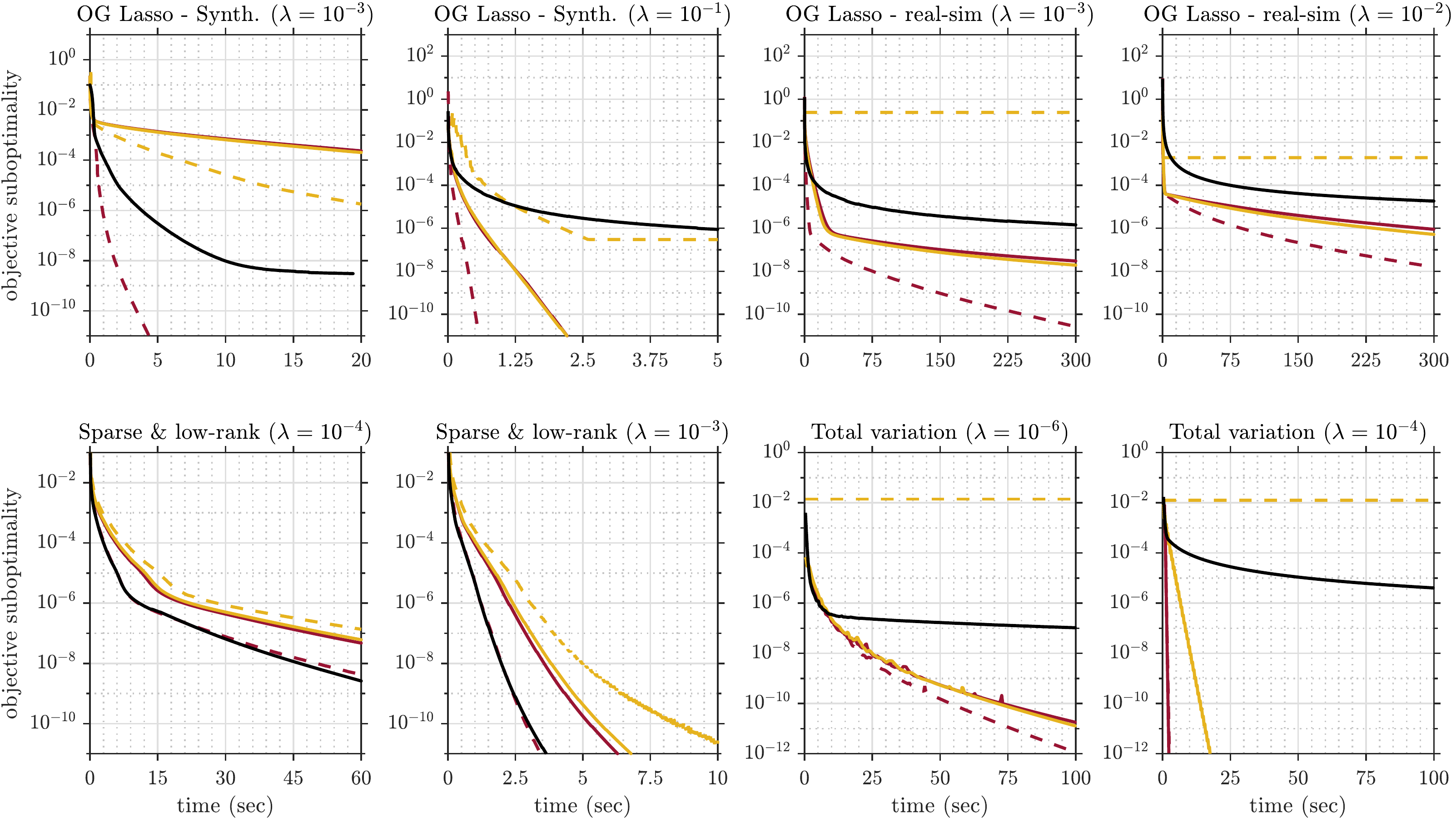}
\vspace{-1em}
\caption{Empirical comparison of 5 algorithms for \Prob with smooth $f$. Dashed lines represent the line-search variants of TOS and PDHG. The performance of \algo is between TOS-LS and PDHG-LS. TOS and PDHG require the knowledge of the smoothness constant, and TOS-LS uses the Lipschitz constant for one of the nonsmooth terms. }
\vspace{-0.5em}
\label{fig:smooth}
\end{center}
\end{figure}

In this subsection, we compare \algo with TOS, PDHG and their line-search variants TOS-LS and PDHG-LS. 
Our experiments are based on the benchmarks described in \citep{pedregosa2018adaptive} and their source code available in COPT Library \citep{copt} under the new BSD License. We implement \algo and investigate its performance on three different problems: 

$\triangleright$ Logistic regression with \emph{overlapping group lasso} penalty:
\begin{equation}
\min_{x \in \R^{n}} \quad \frac{1}{N} \sum_{i=1}^N \log(1 + \exp(-b_i \ip{a_i}{x})) + \lambda \sum_{G \in \mathcal{G}} \sqrt{|G|} \norm{x_G} + \lambda \sum_{H \in \mathcal{H}} \sqrt{|H|} \norm{x_H},
\end{equation}
where $\{({a}_1,{b}_1), \ldots, ({a}_N,{b}_N)\}$ is a given set of training examples, $\mathcal{G}$ and $\mathcal{H}$ are the sets of distinct groups and $| \cdot |$ denotes the cardinality. The model we use (from COPT) considers groups of size $10$ with $2$ overlapping coefficients. In this experiment, we use the benchmarks on synthetic data (dimensions $n=1002$, $N=100$) %
and real-sim dataset \citep{chang2011libsvm} ($n = 20958$, $N = 72309$). 

$\triangleright$ Image recovery with \emph{total variation} penalty:
\begin{equation}
\min_{X \in \R^{m \times n}} \quad \norm{Y - \mathcal{A}(X)}_F^2 + \lambda \sum_{i=1}^m\sum_{j=1}^{n-1}|X_{i,j+1} - X_{i,j}| + \lambda  \sum_{j=1}^n\sum_{i=1}^{m-1}|X_{i+1,j} - X_{i,j}|,
\end{equation}
where $Y$ is a given blurred image and $\mathcal{A}: \R^{m\times n} \to \R^{m\times n}$ is a linear operator (blur kernel). The benchmark in COPT solves this problem for an image of size $153 \times 115$ with a provided blur kernel.

$\triangleright$ Sparse and low-rank matrix recovery via \emph{$\ell_1$ and nuclear-norm} regularizations:
\begin{equation}
\min_{X \in \R^{n \times n}} \quad \frac{1}{N} \sum_{i=1}^N \texttt{huber}(b_i - \ip{A_i}{X}) + \lambda \norm{X}_* + \lambda \norm{X}_{1}.
\end{equation}
We use \texttt{huber} loss. $\{(A_1,b_1), \ldots, (A_N,b_N)\}$ is a given set of measurements and $\norm{X}_{1}$ is the vector $\ell_1$-norm of $X$. The benchmark in COPT considers a symmetric ground truth matrix $X^\natural \in \R^{20 \times 20}$ and noisy synthetic measurements ($N=100$) where $A_i$ has Gaussian iid entries. $b_i = \ip{A_i}{X^\natural} + \omega_i$ where $\omega_i$ is generated from a zero-mean unit variance Gaussian distribution. 

At each problem, we consider two different values for the regularization parameter $\lambda$. We use all methods with their default parameters in the benchmark. For \algo, we discard $\beta$ and tune $\alpha$ by trying the powers of $10$. See the supplementary material for the behavior of the algorithm with different values of $\alpha$.  \Cref{fig:smooth} shows the results of this experiment. In most cases, the performance of \algo is between TOS-LS and PDHG-LS. Remark that TOS-LS is using the extra knowledge of the Lipschitz constant of $h$.  

\subsection{Experiments on Convex Optimization with Nonsmooth \texorpdfstring{$f$}{f}}
\label{sec:experiments-nonsmooth}

\begin{figure}[!ht]
\begin{center}
\includegraphics[width=\linewidth]{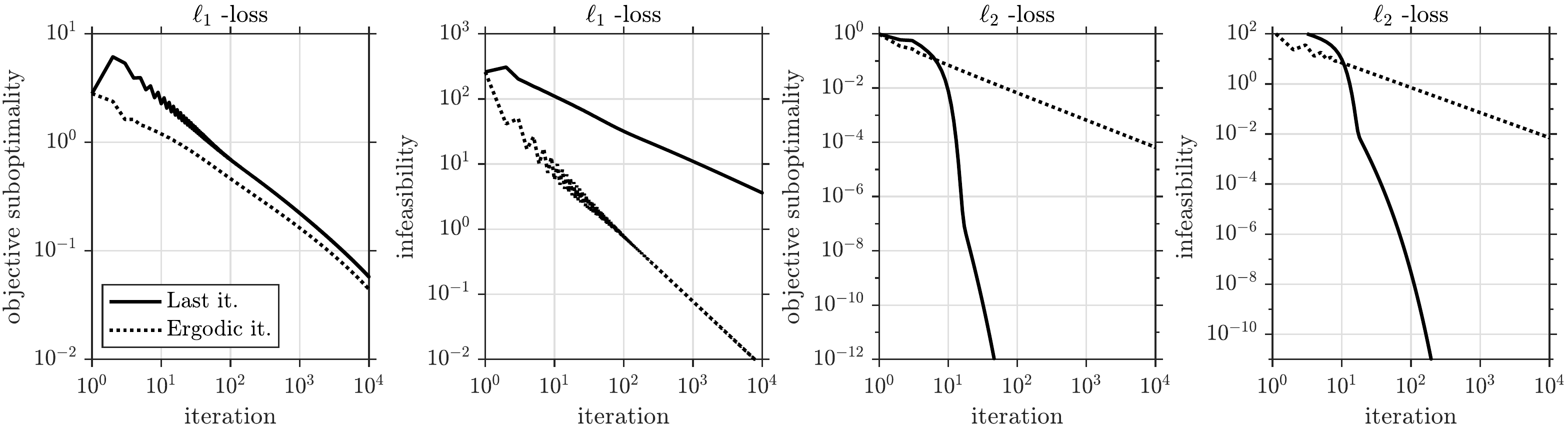}
\vspace{-1.5em}
\caption{Performance of \algo on image impainting and denoising problems with $\ell_1$ and $\ell_2$-loss functions. The empirical rates for $\ell_1$-loss match the guaranteed $\smash{\mathcal{O}(1/\sqrt{t})}$ rate in objective suboptimality and $\smash{\mathcal{O}(1/t)}$ in infeasibility. We observe a locally linear convergence rate for the $\ell_2$-loss.}
\vspace{-0.5em}
\label{fig:nonsmooth}
\end{center}
\end{figure}

We examine the empirical performance of \algo for nonsmooth problems on an image impainting and denoising task from  \citep{Zeng2018,yurtsever2018conditional}. We are given an occluded image (\emph{i.e.}, missing some pixels) of size $517 \times 493$, contaminated with salt and pepper noise of $10\%$ density. We use the following template where data fitting is measured in terms of vector $\ell_p$-norm:
\begin{equation}
\label{eqn:impainting-least-squares}
\min_{\boldsymbol{X} \in \R^{m \times n}} \quad \norm{\mathcal{A}(X) - Y}_p \quad \text{subject to} \quad \norm{X}_\ast \leq \lambda, ~~~ 0\leq X \leq 1,
\end{equation}
where $Y$ is the observed noisy image with missing pixels. This is essentially a matrix completion problem, $\mathcal{A}:\R^{m \times n} \to \R^{m \times n}$ is a linear map that samples the observed pixels in $Y$. In particular, we consider \eqref{eqn:impainting-least-squares} with $p=1$ and $p=2$. 
The $\ell_2$-loss %
is common in practice for matrix completion (often in the least-squares form) but it is not robust against the outliers induced by the salt and pepper noise.  $\ell_1$-loss is known to be more reliable for this task. 

The subgradients in both cases have a fixed norm at all points (note that the subgradients are binary valued for $\ell_1$-loss and unit-norm for $\ell_2$-loss), hence the analytical and the adaptive step-sizes are same up to a constant factor. 

\Cref{fig:nonsmooth} shows the results. 
The empirical rates for $p=1$ roughly match our guarantees in \Cref{thm:convergence-Lipschitz}. We observe a locally linear convergence rate when $\ell_2$-loss is used. Interestingly, the ergodic sequence converges faster than the last iterate for $p=1$ but significantly slower for $p=2$. The runtime of the two settings are approximately the same, with 67 msec per iteration on average. Despite the slower rates, we found $\ell_1$-loss more practical on this problem. A low-accuracy solution obtained by $1000$ iterations on $\ell_1$-loss yields a high quality recovery with PSNR 26.21 dB, whereas the PSNR saturates at 21.15 dB for the $\ell_2$-formulation. See the supplements for the recovered images and more details.

\subsection{An Experiment on Neural Networks}
\label{sec:experiments-mnist}

\begin{figure}[th]
\begin{center}
\includegraphics[width=\textwidth]{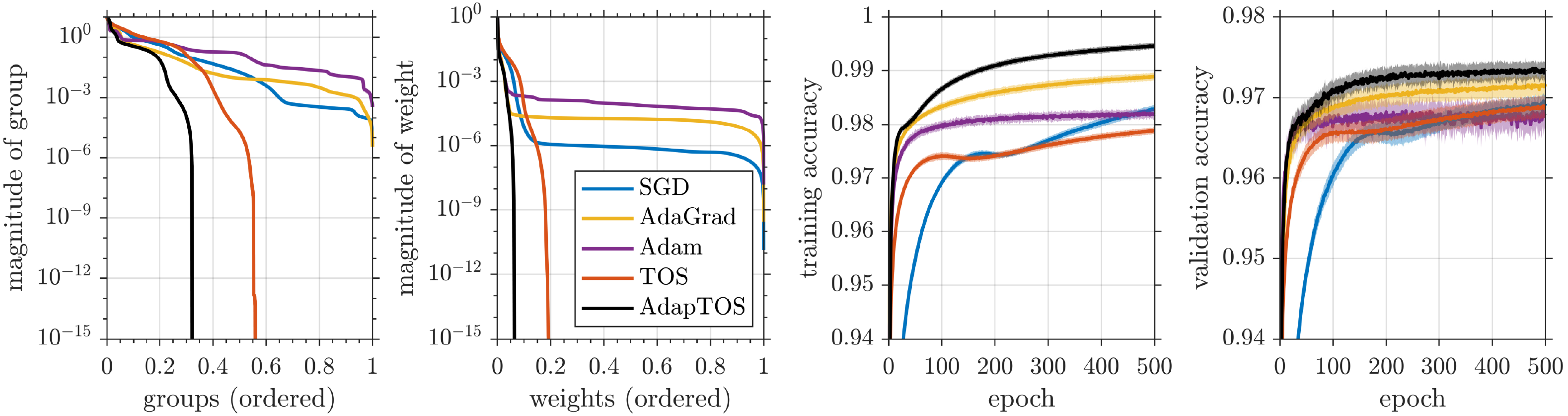} 
\vspace{-1.5em}
\caption{Comparison of methods on training neural networks with group lasso regularization. The outgoing connections of each neuron form a group. The first plot shows the magnitude of weights after 500 epochs. The second plot shows the absolute sum of outgoing weights from each neuron. x-axes are normalized by the total number of weights and neurons in these plots. More than 68\% of the neurons are inactive on the network trained by \algo. The third and fourth plots show the training and validation losses. This experiment is performed with 20 random seeds. The solid lines show the average performance and the shaded area represents $\pm$ standard deviation from the mean.}
\vspace{-0.5em}
\label{fig:nonconvex}
\end{center}
\end{figure}

In this section, we train a regularized deep neural network to test our methods on nonconvex optimization. %
We consider a regularized neural network problem formulation in \citep{scardapane2017group}. This problem involves a fully connected neural network with the standard cross-entropy loss function, a ReLu activation for the hidden layers, and the softmax activation for the output layer. Two regularizers are added to this loss function: The first one is the standard $\ell_1$ regularizer, and the second is the group sparse regularizer where the outgoing connections of each neuron is considered as a group. The goal is to force all outgoing connections from the same neurons to be simultaneously zero, so that we can safely remove the neurons from the network. This is shown as an effective way to obtain compact networks \citep{scardapane2017group}, which is crucial for the deployment of the learned parameters on resource-constrained devices such as smartphones \citep{blalock2020state}.

We reuse the open source implementation (built with Lasagne framework based on Theano) published in \citep{scardapane2017group} under BSD-2 License. We follow their experimental setup and instructions with MNIST database \citep{lecun1998mnist} containing 70k grayscale images ($28 \times 28$) of handwritten digits (split 75/25 into train and test partitions). We train a fully connected neural network with $784$ input features, three hidden layers ($400/300/100$) and 10-dimensional output layer.  Interested readers can find more details on the implementation in the supplementary material or in \citep{scardapane2017group}. 

\citet{scardapane2017group} use SGD and Adam with the subgradient of the overall objective. In contrast, our methods can leverage the prox-operators for the regularizers. \Cref{fig:nonconvex} compares the performance in terms of two measures: the sparsity of the parameters and the accuracy. On the left side, we see the spectrum of weight and neuron magnitudes. The advantage of using prox-operators is outstanding: More than 93\% of the weights are zero and 68\% of neurons are inactive when trained with \algo. In contrast, subgradient based methods can achieve only  approximately sparse solutions.%

The third and the fourth subplots present the training and test accuracies. Remarkably, \algo performs better than the state-of-the-art (both in train and test). Unfortunately, we could not achieve the same performance gain in preliminary experiments with more complex models like ResNet \citep{he2016deep}, where SGD with momentum shines. Interested readers can find the code for these preliminary experiments in the supplements. We leave the technical analysis and a comprehensive examination of \algo for nonconvex problems to a future work. 

\section{Conclusions}
\label{sec:conclusions}

We studied an extension of TOS that permits subgradients and stochastic gradients instead of the gradient step and established convergence guarantees for this extension. Moreover, we proposed an adaptive step-size rule (\algo) for the minimization of a convex function over the intersection of two convex sets. \algo guarantees a nearly optimal $\tilde{\mathcal{O}}(1/\sqrt{t})$ rate on the baseline setting, and it enjoys the faster $\tilde{\mathcal{O}}(1/t)$ rate when the problem is smooth and the solution is in the interior of feasible set. We present numerical experiments on various benchmark problems. The empirical performance of the method is promising. 

We conclude with a short list of open questions and follow-up directions: 
(i) In parallel to the subgradient method, we believe TOS can achieve $\mathcal{O}(1/t)$ rate guarantees in the nonsmooth setting if $f$ is strongly convex. The analysis remains open. 
(ii) The faster rate for \algo on smooth $f$ requires an extra assumption on the location of the solution. We believe this assumption can be removed, and leave this as an open problem.
(iii) We analyzed \algo only for a specific subclass of \Prob in which $g$ and $h$ are indicator functions. Extending this result for the whole class is a valuable question for future study. 

\begin{appendices}

\setcounter{equation}{0}
\renewcommand{\theequation}{S.\arabic{equation}}

\renewcommand{\thefigure}{S.\arabic{figure}}

\renewcommand{\thetheorem}{S.\arabic{theorem}}

\renewcommand{\thelemma}{S.\arabic{lemma}}

\renewcommand{\theremark}{S.\arabic{remark}}

\renewcommand{\theequation}{S.\arabic{equation}}

\section{Preliminaries}

We will use the following standard results in our analysis. 

\begin{lemma} \label{lem:prox-theorem}
Let $f: \mathbb{R}^n \to \mathbb{R} \cup \{+\infty\}$ be a proper closed and convex function. Then, for any $x, u \in \mathbb{R}^n$, the followings are equivalent:
\begin{enumerate}[label=(\roman*),topsep=0pt,itemsep=0ex,partopsep=1ex,parsep=1ex]
\item $u = \prox_f(x)$.
\item $x - u \in \partial f(u)$.
\item $\ip{x-u}{y-u} \leq f(y) - f(u)$ ~for any~ $y \in \mathbb{R}^n$. \\[-0.5em]
\end{enumerate}
\end{lemma}

\begin{corollary}[Firm non-expansivity of the prox-operator] \label{lem:prox-firm-non-expansivity}
Let $f: \mathbb{R}^n \to \mathbb{R} \cup \{+\infty\}$ be a proper closed and convex function. Then, for any $x, u \in \mathbb{R}^n$, the followings hold:
\begin{align*}
\text{(non-expansivity)}\qquad &  \norm{\prox_f(x) - \prox_f(y)} \leq \norm{x - y}. \\[0.5em]
\text{(firm non-expansivity)}\qquad & \norm{\prox_f(x) - \prox_f(y)}^2 \leq \ip{x-y}{\prox_f(x) - \prox_f(y)}.
\end{align*}
\end{corollary}

\section{Fixed Point Characterization}

This appendix presents the proof for \Cref{lem:fixed-point-characterization}. 
This is a straightforward extension of Lemma~2.2 in \citep{davis2017three} to permit subgradients. 
We will use this lemma in the next section to prove the boundedness of $y_t$ in \Cref{alg:three-operator-splitting}.

\subsection{Proof of \texorpdfstring{\Cref{lem:fixed-point-characterization}}{Lemma~\ref{lem:fixed-point-characterization}}}

Define $z =  \prox_{\gamma g}(y)$ and $x = \prox_{\gamma h}(2z - y - \gamma u)$. 
Then, $\texttt{TOS}_\gamma(y, u) := y -z + x$. 

Suppose there exists $u \in \partial f(z)$ such that $\texttt{TOS}(y,u) = y$. 
Then, we must have $z = x$. Moreover, by \Cref{lem:prox-theorem}, we have
\begin{align}
z = \prox_{\gamma g}(y) &\iff y - z \in \gamma \partial g(z), \label{eqn:supp-fpc-a} \\
 \text{and}\quad z = x = \prox_{\gamma h}(2z-y-\gamma u) %
&\iff z-y-\gamma u \in \gamma \partial h(x) . \label{eqn:supp-fpc-b}
\end{align}
By summing up \eqref{eqn:supp-fpc-a} and \eqref{eqn:supp-fpc-b}, we observe
\begin{align}
0 \in \gamma (u+\partial g(z) + \partial h(x)) \implies 0 \in \partial f(z)+ \partial g(z)+\partial h(z) = %
\partial \phi(z),
\end{align}
which proves that $z$ is an optimal solution of \Prob since $\phi$ is convex. 

To prove the reverse direction, suppose $z$ is an optimal solution, i.e., there exists $u \in \partial f(z), v \in \partial g(z), w \in \partial h(z)$ such that $u + v + w = 0$. %
By \Cref{lem:prox-theorem}, we have
\begin{align}
z = \prox_{\gamma g}(y)  &\iff  y-z \in \gamma \partial g(z), \\
 \text{and}\quad  x = \prox_{\gamma h}(2z - y - \gamma u) &\iff 2z-x-y-\gamma u \in \gamma \partial h(x). 
\end{align}
Now, let $y = z + \gamma v$. Then,
\begin{align}
2z-x-y-\gamma u 
= z - x - \gamma (u + v)
= z - x + \gamma w. 
\end{align}
Therefore, we have $z - x + \gamma w \in \partial h(x)$. Again, due to \Cref{lem:prox-theorem}, this means $x = \prox_{\gamma h}(z + \gamma w)$.  
We also know $w \in \partial h(z) \iff z + \gamma w - z \in \partial \gamma h(z) \iff z = \prox_{\gamma h}(z + \gamma w)$. 
However, since $h$ is convex, its prox-operator is unique, hence, $x = z$ and $\texttt{TOS}_\gamma (y,u) = y$.

\section{Boundedness Guarantees}

\begin{theorem}
\label{thm:nonsmooth-y-bounded}
Consider \Prob and employ TOS (\Cref{alg:three-operator-splitting}) with subgradient steps ${u_t \in \partial f(z_t)}$
and a fixed step-size $\gamma = \gamma_0 / \sqrt{T+1}$ for some $\gamma_0 > 0$. Assume that $\norm{u_t} \leq G_f$ for all $t$. Then, 
\begin{align}
\norm{y_{T+1} - y_\star} 
\leq \norm{y_0-y_\star} + 2 \gamma_0 G_f
\end{align}
where $y_\star$ is a fixed point of TOS.
\end{theorem}

\begin{proof}
By \Cref{lem:fixed-point-characterization}, there exists $u_\star \in \partial f (x_\star)$ such that 
\begin{align}
x_\star = \prox_{\gamma g} (y_\star) 
=  \prox_{\gamma h} (2x_\star - y_\star - \gamma u_\star) 
= z_\star. 
\end{align}
We decompose $\norm{y_{t+1} - y_\star}^2$ as
\begin{align}
\norm{y_{t+1} - y_\star}^2 
& = \norm{y_t-z_t+x_t - y_\star + x_\star-x_\star}^2 \notag \\
& = \norm{y_t-z_t - y_\star + x_\star}^2 + \norm{x_t - x_\star}^2 + 2 \ip{x_t - x_\star}{y_t-z_t - y_\star + x_\star}. \label{eqn:ytys2-decomposition-a}
\end{align}
Since $z_t = \prox_{\gamma g} (y_t)$ and $x_\star = \prox_{\gamma g} (y_\star)$, by the firm non-expansivity of the prox-operator, we have
\begin{align} 
\norm{y_t-z_t - y_\star + x_\star}^2
& = \ip{y_t-z_t - y_\star + x_\star}{y_t - y_\star} - \ip{y_t-z_t - y_\star + x_\star}{z_t - x_\star} \notag \\
& \leq \ip{y_t-z_t - y_\star + x_\star}{y_t - y_\star}. \label{eqn:ytzt2-firm-non-expansivity-a}
\end{align}
Similarly, since  $x_t = \prox_{\gamma h} (2z_t - y_t - \gamma u_t)$ and $x_\star = \prox_{\gamma h} (2x_\star - y_\star - \gamma u_\star)$, by the firm non-expansivity of the prox-operator, we have 
\begin{equation} \label{eqn:xt2-firm-non-expansivity-a}
\norm{x_t - x_\star}^2
\leq \ip{x_t - x_\star}{(2z_t - y_t - \gamma u_t) - (2x_\star - y_\star - \gamma u_\star)}.
\end{equation}
By combining \eqref{eqn:ytys2-decomposition-a}, \eqref{eqn:ytzt2-firm-non-expansivity-a} and \eqref{eqn:xt2-firm-non-expansivity-a}, we get
\begin{align}
\norm{y_{t+1} - y_\star}^2 
& \leq \ip{y_t-z_t +x_t - y_\star}{y_t - y_\star} - \gamma \ip{x_t - x_\star}{u_t - u_\star} \notag \\
& = \ip{y_{t+1} - y_\star}{y_t - y_\star} - \gamma \ip{x_t - x_\star}{u_t - u_\star} \notag \\
& = \frac{1}{2} \norm{y_{t+1} - y_\star}^2  + \frac{1}{2} \norm{y_t-y_\star}^2 - \frac{1}{2} \norm{y_{t+1} - y_t}^2 - \gamma \ip{x_t - x_\star}{u_t - u_\star}. \label{eqn:ytys2-decomposition-2}
\end{align}
Since $u_t \in \partial f(z_t)$ and $u_\star \in \partial f(x_\star)$, we have
\begin{align}
- \ip{x_t - x_\star}{u_t - u_\star}
& = - \ip{z_t - x_\star}{u_t - u_\star} - \ip{x_t - z_t}{u_t - u_\star} \notag \\ 
& \leq  - \ip{x_t - z_t}{u_t - u_\star} \notag \\ 
& \leq \frac{1}{2\gamma} \norm{x_t - z_t}^2 + \frac{\gamma}{2} \norm{u_t - u_\star}^2,
\end{align}
where we used Young's inequality in the last line. 
We use this inequality in \eqref{eqn:ytys2-decomposition-2} to obtain
\begin{align}
\norm{y_{t+1} - y_\star}^2 
\leq \norm{y_t-y_\star}^2 + \gamma^2\norm{u_t - u_\star}^2.
\end{align}
If we sum this inequality from $t = 0$ to $T$, we get
\begin{align}
\norm{y_{T+1} - y_\star}^2 
\leq \norm{y_0-y_\star}^2 + \gamma^2 \sum_{\tau = 0}^T \norm{u_\tau - u_\star}^2. 
\end{align}
Finally, due to the bounded subgradients assumption, we have $ \norm{u_\tau - u_\star} \leq 2G_f$, hence
\begin{align}
\norm{y_{T+1} - y_\star}^2 
& \leq \norm{y_0-y_\star}^2 + 4 G_f^2 \gamma^2 (T+1)
= \norm{y_0-y_\star}^2 + 4 G_f^2 \gamma_0^2.
\end{align}
We complete the proof by taking the square-root of both sides,
\begin{align}
\norm{y_{T+1} - y_\star}
\le \sqrt{ \norm{y_0-y_\star}^2 + 4 G_f^2\gamma_0^2} \le \norm{y_0-y_\star} + 2G_f\gamma_0.
\end{align}
\end{proof}

\begin{theorem}
\label{thm:stochastic-y-bounded}
Consider \Probsto and employ TOS (\Cref{alg:three-operator-splitting}) with a fixed step-size $\gamma = \gamma_0 / \sqrt{T+1}$ for some $\gamma_0 > 0$. Suppose we are receiving the update directions $u_t$ from an unbiased stochastic first-order oracle with bounded variance such that
\begin{align}
    \hat{u}_t := \mathbb{E}[u_t | z_t] \in \partial f(z_t) \quad \text{and} \quad \mathbb{E}[\norm{u_t - \hat{u}_t}^2] \leq \sigma^2  ~~ \text{for some $\sigma < +\infty$.}
\end{align}
Assume that $\norm{\hat{u}_t} \leq G_f$ for all $t$. Then, 
\begin{align}
\mathbb{E}[\norm{y_{T+1} - y_\star}] 
\leq \norm{y_0-y_\star} + \gamma_0 (2 G_f + \sigma)
\end{align}
where $y_\star$ is a fixed point of TOS.
\end{theorem}

\begin{proof}
We follow the same steps as in the proof of \Cref{thm:nonsmooth-y-bounded} until \eqref{eqn:ytys2-decomposition-2}:
\begin{align}
\norm{y_{t+1} - y_\star}^2 
& \leq \norm{y_t-y_\star}^2 - \norm{y_{t+1} - y_t}^2 - 2 \gamma \ip{x_t - x_\star}{u_t - u_\star}. \label{eqn:ytys2-decomposition-2-repeat}
\end{align}
Then, we need to take noise into account:
\begin{align}
-\ip{x_t - x_\star}{u_t - u_\star} 
& = - \ip{x_t - z_t}{u_t - u_\star} - \ip{z_t - x_\star}{u_t - \hat{u}_t} - \ip{z_t - x_\star}{\hat{u}_t - u_\star} \notag \\
& \leq - \ip{x_t - z_t}{u_t - u_\star} - \ip{z_t - x_\star}{u_t - \hat{u}_t} . \label{eqn:eqn:stochastic-boundedness-exp-a}
\end{align}
We take the expectation of both sides and get
\begin{align}
- \mathbb{E}[\ip{x_t - x_\star}{u_t-  u_\star}]
& \leq - \mathbb{E}[\ip{x_t - z_t}{u_t -  u_\star}] \notag \\
& \leq \frac{1}{2\gamma}\mathbb{E}[\norm{x_t - z_t}^2] + \frac{\gamma}{2}\mathbb{E}[\norm{ u_\star - u_t}^2] \notag \\
& = \frac{1}{2\gamma}\mathbb{E}[\norm{x_t - z_t}^2] + \frac{\gamma}{2}\mathbb{E}[\norm{ u_\star - \hat{u}_t}^2] + \frac{\gamma}{2} \mathbb{E}[\norm{\hat{u}_t - u_t}^2] \notag \\
& \leq \frac{1}{2\gamma}\mathbb{E}[\norm{x_t - z_t}^2] + 2 \gamma G_f^2 + \frac{\gamma}{2}\sigma^2, \label{eqn:stochastic-boundedness-a}
\end{align}
where the last line holds due to the bounded subgradients and variance assumptions.

Now, we take the expectation of \eqref{eqn:ytys2-decomposition-2-repeat} and substitute \eqref{eqn:stochastic-boundedness-a} into it:
\begin{align}
\mathbb{E}[\norm{y_{t+1} - y_\star}^2]
\leq \mathbb{E}[\norm{y_t-y_\star}^2] + \gamma^2(4 G_f^2 + \sigma^2).
\end{align}
Finally, we sum this inequality over $t=0, 1, \ldots, T$: 
\begin{align}
\mathbb{E}[\norm{y_{T+1} - y_\star}^2]
\leq \norm{y_0-y_\star}^2 + \gamma_0^2(4 G_f^2 + \sigma^2).
\end{align}
By Jensen's inequality, we have $\mathbb{E}[\norm{y_{T+1} - y_\star}]^2 \leq \mathbb{E}[\norm{y_{T+1} - y_\star}^2]$. We finalize the proof by taking the square-root of both sides. 
\end{proof}

Next, we assume that $f$ is $L_f$-smooth instead of Lipschitz continuity. 

\begin{theorem}
Consider \Probsto and suppose $f$ is $L_f$-smooth on $\dom(g)$. Employ TOS (\Cref{alg:three-operator-splitting}) with a fixed step-size $\gamma_t = \gamma = \gamma_0 / \sqrt{T+1}$ for some $\gamma_0 \in [0,\frac{2}{L_f}]$. Suppose we are receiving the update directions $u_t$ from an unbiased stochastic first-order oracle with bounded variance such that
\begin{align}
    \mathbb{E}[u_t | z_t] = \nabla f(z_t) \quad \text{and} \quad \mathbb{E}[\norm{u_t - \nabla f(z_t)}^2] \leq \sigma^2  ~~ \text{for some $\sigma < +\infty$.}
\end{align} 
Then, 
\begin{align}
\mathbb{E}[\norm{y_{T+1} - y_\star}]
\leq \norm{y_0-y_\star} + 2\sigma\sqrt{\frac{\gamma_0}{L_f}},
\end{align}
where $y_\star$ is a fixed point of TOS.
\end{theorem}

\begin{proof}
The proof is similar to the proof of \Cref{thm:stochastic-y-bounded}. We start from \eqref{eqn:ytys2-decomposition-2-repeat} and take the expectation:
\begin{align}
\mathbb{E}[\norm{y_{t+1} - y_\star}^2]
& \leq \mathbb{E}[\norm{y_t-y_\star}^2] - \mathbb{E}[\norm{y_{t+1} - y_t}^2] - 2 \gamma \mathbb{E}[\ip{x_t - x_\star}{u_t - u_\star}].
\label{eqn:eqn:ytys2-decomposition-2-repeat2}
\end{align}
We decompose the last term as follows:
\begin{align}
\mathbb{E}[\ip{x_t - x_\star}{u_\star - u_t}]
& = \mathbb{E}[\ip{x_t - z_t}{u_\star - u_t}] + \mathbb{E}[\ip{z_t - x_\star}{u_\star - \nabla f(z_t)}] + \mathbb{E}[\ip{z_t - x_\star}{\nabla f(z_t) - u_t}] \notag \\
& \leq \mathbb{E}[\ip{x_t - z_t}{u_\star - u_t}] - \frac{1}{L_f} \mathbb{E}[\norm{u_\star - \nabla f(z_t)}^2].
\label{eqn:cancellation1}
\end{align}
where the inequality holds since $f$ is $L_f$-smooth and convex. Moreover, we can bound the inner product term by using Young's inequality as follows:
\begin{align}
\mathbb{E}[\ip{x_t - z_t}{u_\star - u_t}]
& = \mathbb{E}\big[\ip{x_t - z_t}{u_\star - \nabla f(z_t)} + \ip{x_t - z_t}{\nabla f(z_t) - u_t}\big] \notag \\
& \leq \mathbb{E}\Big[\frac{c_1 + c_2}{2}\norm{x_t - z_t}^2 + \frac{1}{2c_1}\norm{u_\star - \nabla f(z_t)}^2 + \frac{1}{2c_2}\norm{\nabla f(z_t) - u_t}^2\Big]
\label{eqn:cancellation2}
\end{align}
for any $c_1,c_2 >0$. We choose $c_1 = L_f/2$, so that the corresponding terms in \eqref{eqn:cancellation1} and \eqref{eqn:cancellation2} cancel out. Combining \eqref{eqn:eqn:ytys2-decomposition-2-repeat2}, \eqref{eqn:cancellation1} and \eqref{eqn:cancellation2}, we get
\begin{align}
\mathbb{E}[\norm{y_{t+1} - y_\star}^2]
& \leq \mathbb{E}[\norm{y_t-y_\star}^2] - \mathbb{E}[\norm{y_{t+1} - y_t}^2] + \gamma (\tfrac{L_f}{2} + c_2) \mathbb{E}[\norm{x_t - z_t}^2] + \gamma\frac{\sigma^2}{c_2} \notag \\
& \leq \mathbb{E}[\norm{y_t-y_\star}^2] + \Big(\gamma \frac{L_f + 2c_2}{2} - 1\Big) \mathbb{E}[\norm{x_t - z_t}^2] + \gamma\frac{\sigma^2}{c_2}.\label{eqn:stochastic-bound-1}
\end{align}
Then, we choose $c_2 = \tfrac{L_f}{2}(\sqrt{T+1} - 1)$. With the condition $\gamma_0 \leq \tfrac{2}{L_f}$, this guarantees
\begin{align}
\gamma \frac{L_f + 2c_2}{2} - 1 
= \frac{\gamma_0}{\sqrt{T+1}} \frac{L_f \sqrt{T+1}}{2} - 1 
\leq  \gamma_0\frac{L_f}{2} - 1  \leq 0.
\end{align}
Returning to \eqref{eqn:stochastic-bound-1}, we now have
\begin{align}
\mathbb{E}[\norm{y_{t+1} - y_\star}^2]
& \leq \mathbb{E}[\norm{y_t-y_\star}^2] + \frac{\gamma_0}{\sqrt{T+1}}\frac{2\sigma^2}{L_f(\sqrt{T+1} - 1)} 
\leq \mathbb{E}[\norm{y_t-y_\star}^2] + \frac{4\gamma_0\sigma^2}{L_f(T+1)}.
\end{align}
Finally, we sum this inequality over $t=0$ to $T$, 
\begin{align}
\mathbb{E}[\norm{y_{T+1} - y_\star}^2]
\leq \norm{y_0-y_\star}^2 + \frac{4 \gamma_0 \sigma^2}{L_f}.
\end{align}
Remark that $\mathbb{E}[\norm{y_{T+1} - y_\star}]^2 \leq \mathbb{E}[\norm{y_{T+1} - y_\star}^2]$. We finish the proof by taking the square-root of both sides. 
\end{proof}

\section{Convergence Guarantees}

This section presents the technical analysis of our main results. 

\subsection{Proof of \texorpdfstring{\Cref{thm:convergence-Lipschitz}}{Theorem~\ref{thm:convergence-Lipschitz}}}

We divide this proof into two parts.

\textbf{Part 1.} In the first part, we show that the sequence generated by TOS satisfies 
\begin{equation}
\label{eqn:nonsmooth-convergence-part1}
\ip{u_t}{x_t - x_\star} + g(z_t) - g (x_\star) + h (x_t) - h (x_\star) \leq \frac{1}{2\gamma} \norm{y_t - x_\star}^2 - \frac{1}{2\gamma} \norm{y_{t+1} - x_\star}^2 - \frac{1}{2\gamma} \norm{y_{t+1} - y_t}^2.
\end{equation}

Since $x_t = \prox_{\gamma h} (2z_t - y_t - \gamma u_t)$, by \Cref{lem:prox-theorem}, we have
\begin{align}
\ip{2z_t - y_t - \gamma u_t - x_t}{x_\star - x_t} \leq \gamma h (x_\star) - \gamma h (x_t).
\end{align}
We rearrange this inequality as follows:
\begin{align}
\ip{u_t}{x_t - x_\star} + h (x_t) - & h (x_\star) 
 \leq \frac{1}{\gamma} \ip{2z_t - y_t -x_t}{x_t - x_\star} \notag \\
& = \frac{1}{\gamma} \ip{z_t - y_t }{z_t - x_\star} + \frac{1}{\gamma} \ip{z_t - y_t }{x_t - z_t} +  \frac{1}{\gamma} \ip{z_t - x_t}{x_t - x_\star} \notag \\
& = \frac{1}{\gamma} \ip{z_t - y_t }{z_t - x_\star} + \frac{1}{\gamma} \ip{y_t + x_t - z_t - x_\star}{z_t - x_t} \notag \\
& = \frac{1}{\gamma} \ip{z_t - y_t }{z_t - x_\star} + \frac{1}{\gamma} \ip{y_{t+1} - x_\star}{y_t - y_{t+1}}.
\end{align}
Then, we use \Cref{lem:prox-theorem} once again (for $\gamma g$) and get
\begin{align}
\ip{u_t}{x_t - x_\star} + g(z_t) - g (x_\star) + h (x_t) - h (x_\star) 
& \leq \frac{1}{\gamma} \ip{y_{t+1} - x_\star}{y_t - y_{t+1}} \notag \\
& \leq  \frac{1}{2\gamma} \norm{y_t - x_\star}^2 - \frac{1}{2\gamma} \norm{y_{t+1} - x_\star}^2 - \frac{1}{2\gamma} \norm{y_{t+1} - y_t}^2.
\end{align}
This completes the first part of the proof. 

\textbf{Part 2.} In the second part, we characterize the convergence rate of $f(\bar{z}_t) + g(\bar{z}_t) + h(\bar{x}_t) - \phi_\star$ to $0$ by using  \eqref{eqn:nonsmooth-convergence-part1}. 
Since $f$ is convex, we have
\begin{align}
\ip{u_t}{x_t - x_\star} 
= \ip{u_t}{z_t - x_\star} - \ip{u_t}{z_t - x_t} 
& \geq f(z_t) - f(x_\star) - \frac{1}{2\gamma} \norm{x_t - z_t}^2 - \frac{\gamma}{2} \norm{u_t}^2. \label{eqn:nonsmooth-convergence-part2-a}
\end{align}
By combining \eqref{eqn:nonsmooth-convergence-part1} and \eqref{eqn:nonsmooth-convergence-part2-a}, we obtain
\begin{align}
f(z_t)  + g(z_t) + h (x_t) - \phi_\star \leq \frac{1}{2\gamma} \norm{y_t - x_\star}^2 - \frac{1}{2\gamma} \norm{y_{t+1} - x_\star}^2  + \frac{\gamma}{2} \norm{u_t}^2.
\end{align}
We sum this inequality over $t=0$ to $T$:
\begin{align} 
\sum_{\tau=0}^T \Big( f(z_\tau) + g(z_\tau) + h (x_\tau) - \phi_\star \Big) 
& \leq \frac{1}{2\gamma} \norm{y_0 - x_\star}^2 + \frac{\gamma}{2} \sum_{\tau=0}^T  \norm{u_\tau}^2 
\leq \frac{1}{2\gamma} \norm{y_0 - x_\star}^2 + \frac{\gamma_0}{2} G_f^2 \sqrt{T+1},
\label{eqn:nonsmooth-convergence-part2-b}
\end{align}
where the second inequality holds due to the bounded subgradients assumption.
Finally, we divide both sides by $(T+1)$ and use Jensen's inequality:
\begin{align} 
f(\bar{z}_t) + g(\bar{z}_t) + h (\bar{x}_t) - \phi_\star 
\leq \frac{1}{2\sqrt{T+1}} \left(\frac{1}{\gamma_0} \norm{y_0 - x_\star}^2 + \gamma_0 G_f^2 \right). \label{eqn:nonsmooth-convergence-part2-e}
\end{align}

\subsection{Proof of \texorpdfstring{\Cref{thm:stochastic}}{Theorem~\ref{thm:stochastic}}}

The proof is similar to the proof of \Cref{thm:convergence-Lipschitz}. We will only discuss the different steps. 
Part~1 of the proof is the same, \textit{i.e.,} \eqref{eqn:nonsmooth-convergence-part1} is still valid. 

We need to consider the randomness of the gradient estimator in the second part. To this end, we modify \eqref{eqn:nonsmooth-convergence-part2-a} as:
\begin{align}
\mathbb{E} [\ip{u_t}{x_t - x_\star}]
& = \mathbb{E}[\ip{\hat{u}_t}{z_t - x_\star}] + \mathbb{E}[\ip{u_t - \hat{u}_t}{z_t - x_\star}] - \mathbb{E}[\ip{u_t}{z_t - x_t}] \notag \\
& \geq \mathbb{E}[f(z_t) - f(x_\star)] - \mathbb{E}[\ip{u_t}{z_t - x_t}] \notag \\
& \geq \mathbb{E}[f(z_t) - f(x_\star)] - \frac{1}{2\gamma}\mathbb{E}[\norm{z_t - x_t}^2] - \frac{\gamma}{2} \mathbb{E}[\norm{u_t}^2] \notag \\
& \geq \mathbb{E}[f(z_t) - f(x_\star)] - \frac{1}{2\gamma}\mathbb{E}[\norm{z_t - x_t}^2] - \frac{\gamma}{2} (G_f^2 + \sigma^2),
\label{eqn:nonsmooth-convergence-part2-a2}
\end{align}
where the last line holds since 
\begin{align}
\mathbb{E}[\norm{u_t}^2] 
& = \mathbb{E}[\norm{u_t - \hat{u}_t + \hat{u}_t}^2] \\
& = \mathbb{E}[\norm{u_t - \hat{u}_t}^2] + \mathbb{E}[\norm{\hat{u}_t}^2] + 2\mathbb{E}[\ip{u_t - \hat{u}_t}{\hat{u}_t}] 
\leq \sigma^2 + G_f^2.
\end{align}

Now, we take the expectation of \eqref{eqn:nonsmooth-convergence-part1} and substitute \eqref{eqn:nonsmooth-convergence-part2-a2} into it:
\begin{equation}
\mathbb{E}[f(z_t) + g(z_t) + h(x_t)] - \phi_\star \leq \frac{1}{2\gamma} \mathbb{E}[\norm{y_t - x_\star}^2] - \frac{1}{2\gamma} \mathbb{E}[\norm{y_{t+1} - x_\star}^2] + \frac{\gamma}{2} (\sigma^2 + G_f^2).
\end{equation}
We sum this inequality from $t=0$ to $T$ and divide both sides by $T+1$. Then, we use Jensen's inequality and get
\begin{align} 
\mathbb{E} [f(\bar{z}_T) + g(\bar{z}_T) + h (\bar{x}_T)] - \phi_\star 
\leq \frac{1}{2\sqrt{T+1}} \left(\frac{1}{\gamma_0} \norm{y_0 - x_\star}^2 + \gamma_0 (\sigma^2 + G_f^2)  \right).
\end{align}

\subsection{TOS for the Smooth and Stochastic Setting (\texorpdfstring{\Cref{rem:smooth-stochastic}}{Remark~\ref{rem:smooth-stochastic}})}
\label{sec:stochastic-smooth-TOS}

\begin{theorem}
Consider \Probsto and suppose $f$ is $L_f$-smooth on $\dom(g)$. Employ TOS (\Cref{alg:three-operator-splitting}) with a fixed step-size $\gamma = \gamma_0 / \sqrt{T+1}$ for some $\gamma_0 \in [0,\frac{1}{2L_f}]$. Suppose we are receiving the update directions $u_t$ from an unbiased stochastic first-order oracle with bounded variance such that
\begin{align}
    \mathbb{E}[u_t | z_t] = \nabla f(z_t) \quad \text{and} \quad \mathbb{E}[\norm{u_t - \nabla f(z_t)}^2] \leq \sigma^2  ~~ \text{for some $\sigma < +\infty$.}
\end{align} 
Then, the following guarantees hold: 
\begin{align} 
\mathbb{E} [f(\bar{x}_T) + g(\bar{z}_T) + h (\bar{x}_T)] - \phi_\star 
\leq \frac{1}{\sqrt{T+1}} \left(\frac{D^2}{2\gamma_0} + \gamma_0 \sigma^2  \right), ~~~ \text{where} ~~~ D = \norm{y_0 - x_\star}. 
\end{align}
\end{theorem}

\begin{proof}
The proof is similar to \Cref{thm:convergence-Lipschitz}.  \eqref{eqn:nonsmooth-convergence-part1} still holds. 
We modify \eqref{eqn:nonsmooth-convergence-part2-a} as follows (similar to \eqref{eqn:nonsmooth-convergence-part2-a2}):
\begin{align}
\mathbb{E} [\ip{u_t}{x_t - x_\star}]
& \geq \mathbb{E}[f(z_t) - f(x_\star)] + \mathbb{E}[\ip{\nabla f(z_t) - u_t}{z_t - x_t}] - \mathbb{E}[\ip{\nabla f(z_t)}{z_t - x_t}] \notag \\
& \geq \mathbb{E}[f(x_t) - f(x_\star)] - \frac{1}{4\gamma} \mathbb{E}[\norm{z_t - x_t}^2] - \gamma \mathbb{E}[\norm{\nabla f(z_t) - u_t}^2] - \frac{L_f}{2} \norm{x_t - z_t}^2 \notag \\
& \geq \mathbb{E}[f(x_t) - f(x_\star)] - \frac{1+2\gamma L_f}{4\gamma} \mathbb{E}[\norm{y_{t+1} - y_t}^2] - \gamma \sigma^2 . \label{eqn:smooth-convergence-part2-a2}
\end{align}
We take the expectation of \eqref{eqn:nonsmooth-convergence-part1} and replace \eqref{eqn:smooth-convergence-part2-a2} into it
\begin{align}
\mathbb{E}[f(x_t) + g(z_t) + h(x_t)] - \phi_\star 
& \leq \frac{1}{2\gamma} \norm{y_t - x_\star}^2 - \frac{1}{2\gamma} \norm{y_{t+1} - x_\star}^2 + \frac{2\gamma L_f - 1}{4\gamma} \mathbb{E}[\norm{y_{t+1} - y_t}^2] + \gamma \sigma^2 \notag \\
& \leq \frac{1}{2\gamma} \norm{y_t - x_\star}^2 - \frac{1}{2\gamma} \norm{y_{t+1} - x_\star}^2 + \gamma \sigma^2, \label{eqn:tos-smooth-stochastic-b}
\end{align}
where the second line holds since we choose $\gamma_0 \in [0,\frac{1}{2L_f}]$. 

We sum \eqref{eqn:tos-smooth-stochastic-b} from $t=0$ to $T$ and divide both sides by $T+1$. We complete the proof by using Jensen's inequality. 
\end{proof}

\section{Convergence Guarantees for \algo}

In this section, we focus on \ProbAdapTos, an important subclass of \Prob where $g$ and $h$ are indicator functions. 
In this setting, TOS performs the following steps iteratively for $t = 0,1, \ldots$:
\begin{align}
z_{t} & = \proj_{\mathcal{G}} (y_t) \label{eqn:adaptos-step-1} \\
x_{t} & = \proj_{\mathcal{H}} (2z_t - y_t - \gamma_t u_t) \label{eqn:adaptos-step-2} \\
y_{t+1} & = y_t - z_t + x_t, \label{eqn:adaptos-step-3}
\end{align}
where $\gamma_t$ at line \eqref{eqn:adaptos-step-2} is chosen according to the adaptive step-size rule \eqref{eqn:adaptos-step-size}, \emph{i.e.},
\begin{align}
    \gamma_t = \frac{\alpha}{\sqrt{\beta+\sum_{\tau=0}^{t-1} \norm{u_\tau}^2}} \quad \text{for some $\alpha, \beta > 0$.}
\end{align}

The following lemmas are useful in the analysis. 

\begin{lemma}[Lemma~A.2 in \citep{levy2017online}] 
\label{lem:Levy2017}
Let $f: \mathbb{R}^n \to \mathbb{R}$ be a $L_f$-smooth function and let $x_\star \in \arg\min_{x \in \mathbb{R}^n} f(x)$. Then,
\begin{align*}
    \norm{\nabla f(x)}^2 \leq 2 L_f \big(f(x) - f(x_\star)\big), \qquad \forall x \in \mathbb{R}^n.
\end{align*}
\end{lemma}

\begin{lemma}[Lemma~9 in \citep{bach2019universal}] %
\label{lem:Bach2019a}
For any non-negative numbers $a_0, \ldots, a_t \in [0, a]$, and $\beta \geq 0$
\begin{align*}
\sum_{i=0}^t \frac{a_i}{\sqrt{\beta + \sum_{j=0}^{i-1} a_j}} 
\leq \frac{2a}{\sqrt{\beta}} + 3\sqrt{a} + 3 \sqrt{\beta + \sum_{i=0}^{t-1}a_i}.
\end{align*}
\end{lemma}

\begin{lemma}[Lemma~10 in \citep{bach2019universal}] %
\label{lem:Bach2019b}
For any non-negative numbers $a_0, \ldots, a_t \in [0, a]$, and $\beta \geq 0$
\begin{align*}
\sum_{i=0}^t \frac{a_i}{\beta + \sum_{j=0}^{i-1} a_j} 
\leq 2 + \frac{4a}{\beta} + 2\log\left(1 + \sum_{i=0}^{t-1} \frac{a_i}{\beta} \right).
\end{align*}
\end{lemma}

\begin{corollary}\label{cor:adaptive-bounds}
Suppose $\|u_t\| \leq G$ for all $t$. Then, the following relations hold for \algo:
\begin{align*}
(i). ~~ & \sum_{\tau=0}^t \gamma_\tau \norm{u_\tau}^2 
= \alpha \sum_{\tau=0}^t \frac{\norm{u_\tau}^2}{\sqrt{\beta + \sum_{j=0}^{\tau-1} \norm{u_j}^2}} 
\leq \alpha \left(\frac{2G^2}{\sqrt{\beta}} + 3G + 3\sqrt{\beta + G^2t}\right) \\
(ii). ~~ & \sum_{\tau=0}^t \gamma_\tau^2\norm{u_\tau}^2 
= \alpha^2 \sum_{\tau=0}^t \frac{\norm{u_\tau}^2}{\beta + \sum_{j=0}^{\tau-1} \norm{u_j}^2}
\leq \alpha^2 \left( 2 + \frac{4G^2}{\beta} + 2\log\Big(1 + \frac{G^2}{\beta}t \Big) \right) \\
(iii). ~~ & \sum_{\tau=0}^t \gamma_\tau\norm{u_\tau}
= \sum_{\tau=0}^t \sqrt{\gamma_\tau^2\norm{u_\tau}^2} 
\leq \sqrt{(t+1)  \textstyle \sum_{\tau=0}^t\gamma_\tau^2\norm{u_\tau}^2} 
\leq \alpha \sqrt{t+1} \sqrt{2 + \frac{4G^2}{\beta} + 2\log\Big(1 + \frac{G^2}{\beta}t \Big)}
\end{align*}
\end{corollary}

\subsection{Proof of \texorpdfstring{\Cref{thm:adaptos-nonsmooth}}{Theorem~\ref{thm:adaptos-nonsmooth}}}

First, we will bound the growth rate of $\norm{y_{t+1} - x_\star}$. We decompose $\norm{y_{t+1} - x_\star}^2$ as
\begin{align}
\norm{y_{t+1} - x_\star}^2 
= \norm{y_t-z_t+x_t - x_\star}^2 
= \norm{y_t-z_t}^2 + \norm{x_t - x_\star}^2 + 2 \ip{x_t - x_\star}{y_t-z_t}. \label{eqn:ytys2-decomposition}
\end{align}
Since $z_t = \proj_{\mathcal{G}} (y_t)$ and $x_\star \in \mathcal{G}$, we have
\begin{align} 
\norm{y_t-z_t}^2
= \ip{y_t-z_t}{y_t - x_\star} + \ip{y_t-z_t}{x_\star - z_t}
\leq \ip{y_t-z_t}{y_t - x_\star}. \label{eqn:ytzt2-firm-non-expansivity}
\end{align}
Similarly, since  $x_t = \proj_{\mathcal{H}} (2z_t - y_t - \gamma_t u_t)$ and $x_\star \in \mathcal{H}$, by the firm non-expansivity, we have
\begin{equation} \label{eqn:xt2-firm-non-expansivity}
\norm{x_t - x_\star}^2
\leq \ip{x_t - x_\star}{2z_t - y_t - \gamma_t u_t - x_\star}.
\end{equation}
By combining 
\eqref{eqn:ytys2-decomposition}, \eqref{eqn:ytzt2-firm-non-expansivity} and \eqref{eqn:xt2-firm-non-expansivity}, we get
\begin{align}
\norm{y_{t+1} - x_\star}^2 
& \leq \ip{y_t-z_t +x_t - x_\star}{y_t - x_\star} - \gamma_t\ip{x_t - x_\star}{u_t} \notag \\
& = \ip{y_{t+1} - x_\star}{y_t - x_\star} -  \gamma_t \ip{x_t - x_\star}{u_t} \notag \\
& = \frac{1}{2} \norm{y_{t+1} - x_\star}^2  + \frac{1}{2} \norm{y_t-x_\star}^2 - \frac{1}{2} \norm{y_{t+1} - y_t}^2 - \gamma_t \ip{x_t - x_\star}{  u_t}.  \label{eqn:ytys2-decomposition-2ab}
\end{align}
Now, we rearrange \eqref{eqn:ytys2-decomposition-2ab} as follows:
\begin{align}
\norm{y_{t+1} - x_\star}^2 
& \leq \norm{y_t-x_\star}^2 - \norm{y_{t+1} - y_t}^2 + 2\gamma_t\ip{ u_t}{x_\star - x_t} \notag \\
& = \norm{y_t-x_\star}^2 - \norm{y_{t+1} - y_t}^2 + 2\gamma_t\ip{u_t}{x_\star-z_t} + 2\gamma_t  \ip{u_t}{z_t - x_t } \notag \\
& \leq \norm{y_t-x_\star}^2 + 2\gamma_t\ip{u_t}{x_\star-z_t} + \gamma_t^2\norm{u_t}^2 \notag \\
& \leq \norm{y_t-x_\star}^2 + 2\gamma_t \norm{u_t} \norm{z_t - x_\star} + \gamma_t^2\norm{u_t}^2 \notag \\
& \leq \norm{y_t-x_\star}^2 + 2\gamma_t \norm{u_t} \norm{y_t - x_\star} + \gamma_t^2\norm{u_t}^2 
= (\norm{y_t-x_\star} + \gamma_t \norm{u_t})^2, \label{eqn:adaptos-nonsmooth-a}
\end{align}
where we use non-expansivity of the projection operator in the last line: $\norm{z_t-x_\star} = \norm{\proj_{\mathcal{G}}(y_t)-x_\star} \leq \norm{y_t - x_\star}$. 

Next, we take the square root of both sides and use \Cref{cor:adaptive-bounds} to get
\begin{align}
\norm{y_{t+1} - x_\star} 
& \leq \norm{y_t-x_\star} + \gamma_t \norm{u_t} \notag \\
& \leq \norm{y_0-x_\star} + \sum_{\tau=0}^t \gamma_\tau \norm{u_\tau} \notag \\
& \leq \norm{y_0-x_\star} + \alpha \sqrt{t+1} \sqrt{2 + \tfrac{4G_f^2}{\beta} + 2\log\big(1 + \tfrac{G_f^2}{\beta}t \big)}.
\label{eqn:adaptos-nonsmooth-a-a}
\end{align}
Now, we can derive a bound on the infeasibility as follows:
\begin{multline}
\dist(\bar{z}_t,\mathcal{H}) 
\leq \norm{\bar{x}_t - \bar{z}_t} = \frac{1}{t+1}\norm{\sum_{\tau=0}^t (x_\tau - z_\tau)} = \frac{1}{t+1}\norm{y_{t+1} - y_0}
\leq \frac{1}{t+1} \left(\norm{y_{t+1} - x_\star} + \norm{y_0 - x_\star} \right) \\
\leq \frac{1}{t+1} \bigg(2\norm{y_0-x_\star} +\alpha \sqrt{t+1} \sqrt{2 + \tfrac{4G_f^2}{\beta} + 2\log\big(1 + \tfrac{G_f^2}{\beta}t \big)} \bigg). \label{eqn:adaptos-nonsmooth-a-3}
\end{multline}
Next, we prove convergence in  objective value. 
Define $s_t = \sum_{\tau=0}^t \gamma_t$ and $\tilde{z}_t = \frac{1}{s_t} \sum_{\tau=0}^t \gamma_t z_t$. Since $f$ is convex, by Jensen's inequality, 
\begin{align}
f(\tilde{z}_t) - f_\star 
\leq \frac{1}{s_t}\sum_{\tau=0}^t \gamma_\tau \left(f(z_\tau) - f_\star\right) 
\leq \frac{1}{s_t}\sum_{\tau=0}^t \gamma_\tau \ip{u_t}{z_\tau - x_\star}.\label{eqn:adaptos-nonsmooth-b}
\end{align}
From \eqref{eqn:adaptos-nonsmooth-a}, we have
\begin{align}
\gamma_t\ip{u_t}{z_t - x_\star}
& \leq \frac{1}{2}\norm{y_t-x_\star}^2 - \frac{1}{2}\norm{y_{t+1} - x_\star}^2   + \frac{1}{2}\gamma_t^2\norm{u_t}^2. \label{eqn:adaptos-nonsmooth-c}
\end{align}
If we substitute \eqref{eqn:adaptos-nonsmooth-c} into \eqref{eqn:adaptos-nonsmooth-b}, we obtain 
\begin{align}
f(\tilde{z}_t) - f_\star 
& \leq \frac{1}{2s_t} \bigg(\norm{y_0-x_\star}^2 + \sum_{\tau=0}^t\gamma_\tau^2\norm{u_\tau}^2 \bigg) \notag \\
& \leq \frac{1}{2s_t} \bigg(\norm{y_0-x_\star}^2 + \alpha^2 \left( 2 + \tfrac{4G_f^2}{\beta} + 2\log\big(1 + \tfrac{G_f^2}{\beta}t \big) \right) \bigg) 
\label{eqn:adaptos-nonsmooth-d}
\end{align}
where the second line comes from \Cref{cor:adaptive-bounds}. %
Finally, we note that 
\begin{align}
s_t = \sum_{\tau=0}^t \gamma_\tau
= \sum_{\tau=0}^t \frac{\alpha}{\sqrt{\beta + \sum_{j=0}^{\tau-1} \norm{u_j}^2}} 
\geq \sum_{\tau=0}^t \frac{\alpha}{ \sqrt{\beta + G_f^2 t}} 
= \frac{\alpha (t+1)}{\sqrt{\beta + G_f^2 t}} 
\geq \frac{\alpha (t+1)}{\sqrt{\beta} + G_f \sqrt{t}}. 
\label{eqn:adaptos-nonsmooth-e}
\end{align}
We complete the proof by using \eqref{eqn:adaptos-nonsmooth-e} in \eqref{eqn:adaptos-nonsmooth-d}:
\begin{align}
f(\tilde{z}_t) - f_\star 
\leq  \left(\frac{G_f}{\sqrt{t+1}} + \frac{\sqrt{\beta}}{t+1} \right) \bigg(\frac{1}{2\alpha}\norm{y_0-x_\star}^2 + \alpha \left( 1 + \tfrac{2G_f^2}{\beta} + \log\big(1 + \tfrac{G_f^2}{\beta}t \big) \right) \bigg).
\end{align}

\subsection{Proof of \texorpdfstring{\Cref{thm:adaptos-smooth}}{Theorem~\ref{thm:adaptos-smooth}}}

As in the proof of \Cref{thm:adaptos-nonsmooth}, our first goal is to bound $\norm{y_{t+1} - y_\star}$. We start from \eqref{eqn:adaptos-nonsmooth-a}:
\begin{align}
\norm{y_{t+1} - x_\star}^2 
& \leq \norm{y_t-x_\star}^2 + 2\gamma_t\ip{u_t}{x_\star-z_t} + \gamma_t^2\norm{u_t}^2. \label{eqn:adaptos-smooth-a}
\end{align}
By assumption $f$ is convex and the solution lies in the interior of the feasible set. Hence, $\ip{u_t}{x_\star-z_t} \leq 0$ and 
\begin{align}
\norm{y_{t+1} - x_\star}^2 
& \leq \norm{y_t-x_\star}^2 + \gamma_t^2\norm{u_t}^2 
\leq \norm{y_0-x_\star}^2 + \sum_{\tau=0}^t \gamma_\tau^2\norm{u_\tau}^2 . \label{eqn:adaptive-smooth-boundedness-pre}
\end{align}
By using \Cref{cor:adaptive-bounds}, this leads to
\begin{align}
\norm{y_{t+1} - x_\star}^2 
& \leq \underbrace{\norm{y_0 - x_\star}^2 + \alpha^2 \left( 2 + \tfrac{4G_f^2}{\beta} + 2\log\big(1 + \tfrac{G_f^2}{\beta}t \big) \right)}_{:=D_t^2}.
\label{eqn:adaptive-smooth-boundedness}
\end{align}
We take the square-root of both sides to obtain $\norm{y_{t+1} - x_\star} \leq D_t$. 
This proves that $\norm{y_t - x_\star}$ is bounded by a logarithmic growth. Similar to \eqref{eqn:adaptos-nonsmooth-a-3}, we can use this bound to prove convergence to a feasible point:
\begin{align}
\dist(\bar{z}_t,\mathcal{H}) 
\leq \norm{\bar{x}_t - \bar{z}_t} %
 & \leq \frac{1}{t+1} \Big(\norm{y_{t+1} - x_\star} + \norm{y_0 - x_\star} \Big) 
 \leq \frac{1}{t+1} \Big(D_t + \norm{y_0 - x_\star} \Big) \notag \\
 & \leq \frac{1}{t+1} \bigg(2\norm{y_0 - x_\star} + \alpha \sqrt{ 2 + \tfrac{4G_f^2}{\beta} + 2\log\Big(1 + \tfrac{G_f^2}{\beta}t \Big)} \bigg).
\end{align}
Next, we analyze the objective suboptimality. 
From \eqref{eqn:adaptos-smooth-a}, we have
\begin{align}
\ip{u_t}{z_t - x_\star}
& \leq \frac{1}{2\gamma_t}\norm{y_t-x_\star}^2 - \frac{1}{2\gamma_t} \norm{y_{t+1} - x_\star}^2   + \frac{\gamma_t}{2}\norm{u_t}^2.  \label{eqn:adaptos-smooth-b}
\end{align}
Then, since $f$ is convex, by using Jensen's inequality and \eqref{eqn:adaptos-smooth-b}, we get
\begin{align}
\Phi_t :&= 
\frac{1}{t+1}\sum_{\tau=0}^t \left(f(z_\tau) - f_\star\right) \notag \\
& \leq \frac{1}{t+1} \sum_{\tau=0}^t \ip{u_t}{z_t - x_\star} \notag \\
& \leq \frac{1}{2(t+1)} \bigg( \frac{1}{\gamma_0}\norm{y_0-x_\star}^2 + \underbrace{\sum_{\tau=1}^t \Big( \frac{1}{\gamma_\tau} - \frac{1}{\gamma_{\tau-1}} \Big) \norm{y_\tau - x_\star}^2}_{(*)} + \sum_{\tau=0}^t  \gamma_\tau\norm{u_\tau}^2 \bigg).
\label{eqn:adaptos-smooth-c}
\end{align}
Now, we focus on $(*)$. By using \eqref{eqn:adaptive-smooth-boundedness-pre}, we get
\begin{align}
(*) & \leq \sum_{\tau=1}^t \Big( \frac{1}{\gamma_\tau} - \frac{1}{\gamma_{\tau-1}} \Big) \big( \norm{y_0 - x_\star}^2 + \sum_{j=0}^{\tau-1} \gamma_j^2\norm{u_j}^2 \big) \notag \\
& = \Big( \frac{1}{\gamma_t}-\frac{1}{\gamma_0}\Big) \norm{y_0 - x_\star}^2 + \sum_{\tau=1}^t \frac{1}{\gamma_\tau} \sum_{j=0}^{\tau-1} \gamma_j^2\norm{u_j}^2 -  \sum_{\tau=1}^t \frac{1}{\gamma_{\tau-1}} \sum_{j=0}^{\tau-1} \gamma_j^2\norm{u_j}^2 \notag \\
& = \Big( \frac{1}{\gamma_t}-\frac{1}{\gamma_0}\Big) \norm{y_0 - x_\star}^2 + \sum_{\tau=1}^t \frac{1}{\gamma_\tau} \sum_{j=0}^{\tau-1} \gamma_j^2\norm{u_j}^2 -  \sum_{\tau=0}^{t-1} \frac{1}{\gamma_{\tau}} \sum_{j=0}^{\tau} \gamma_j^2\norm{u_j}^2 \notag \\
& = \Big( \frac{1}{\gamma_t}-\frac{1}{\gamma_0}\Big) \norm{y_0 - x_\star}^2 + \sum_{\tau=1}^t \frac{1}{\gamma_\tau} \sum_{j=0}^{\tau} \gamma_j^2\norm{u_j}^2 - \sum_{\tau=1}^t \gamma_\tau \norm{u_\tau}^2 -\sum_{\tau=0}^{t-1} \frac{1}{\gamma_{\tau}} \sum_{j=0}^{\tau} \gamma_j^2\norm{u_j}^2 \notag \\
& = \Big( \frac{1}{\gamma_t}-\frac{1}{\gamma_0}\Big) \norm{y_0 - x_\star}^2 + \sum_{\tau=0}^t \frac{1}{\gamma_\tau} \sum_{j=0}^{\tau} \gamma_j^2\norm{u_j}^2 - \sum_{\tau=0}^t \gamma_\tau \norm{u_\tau}^2 -\sum_{\tau=0}^{t-1} \frac{1}{\gamma_{\tau}} \sum_{j=0}^{\tau} \gamma_j^2\norm{u_j}^2 \notag \\
& = \Big( \frac{1}{\gamma_t}-\frac{1}{\gamma_0}\Big) \norm{y_0 - x_\star}^2 + \frac{1}{\gamma_t} \sum_{j=0}^{t} \gamma_j^2\norm{u_j}^2 - \sum_{\tau=0}^t \gamma_\tau \norm{u_\tau}^2.
\end{align}
We substitute this back into \eqref{eqn:adaptos-smooth-c} and obtain
\begin{align}
\Phi_t
& \leq \frac{1}{2(t+1)} \frac{1}{\gamma_t} \Big(\norm{y_0-x_\star}^2 + \sum_{j=0}^t \gamma_j^2 \norm{u_j}^2 \Big)
\leq \frac{D_t^2}{2\gamma_t(t+1)}
\end{align}
where $D_t^2$ is defined in \eqref{eqn:adaptive-smooth-boundedness}.

By the definition of $\gamma_t$ we get
\begin{align}
\Phi_t
& \leq \frac{D_t^2}{2\gamma_t(t+1)} 
= \frac{D_t^2}{2\alpha(t+1)} \sqrt{\beta + \sum_{\tau=0}^{t-1}\norm{u_\tau}^2} 
\leq \frac{D_t^2}{2\alpha(t+1)} \sqrt{\beta + \sum_{\tau=0}^{t}\norm{u_\tau}^2}.
\label{eqn:adaptos-smooth-g}
\end{align}
By \Cref{lem:Levy2017}, we have
\begin{align}
\sum_{\tau = 0}^t \|u_\tau\|^2 
\leq 2L_f \sum_{\tau = 0}^t \big(f(z_\tau) - f_\star\big)
= 2L_f (t+1) \Phi_t. 
\end{align}
We place this back into \eqref{eqn:adaptos-smooth-g}, take the square of both sides, and rearrange the inequality as follows:
\begin{align}
\frac{4\alpha^2(t+1)^2}{D_t^4}\Phi_t^2
& \leq 2L_f(t+1)\Phi_t + \beta.
\end{align}
This is a second order inequality of $\Phi_t$. By solving this inequality, we get
\begin{align}
\Phi_t 
& \leq \frac{1}{2(t+1)} \left( \Big(\frac{D_t^2}{\alpha}\Big)^2 L_f + \frac{D_t^2}{\alpha}\sqrt{\beta}\right). 
\end{align}
Finally, we note 
$f(\bar{z}_t) - f_\star \leq \Phi_t$ by Jensen's inequality.  

\subsection{Proof of \texorpdfstring{\Cref{thm:adaptos-stochastic}}{Theorem~\ref{thm:adaptos-stochastic}}}

Once again we start from \eqref{eqn:adaptos-nonsmooth-a} and take the expectation of both sides:
\begin{align}
\mathbb{E}[\norm{y_{t+1} - x_\star}^2]
& \leq \mathbb{E}[\norm{y_t-x_\star}^2] + \mathbb{E}[2\gamma_t\ip{u_t}{x_\star-z_t}] + \mathbb{E}[\gamma_t^2\norm{u_t}^2] \notag \\
& \leq \mathbb{E}[\norm{y_t-x_\star}^2] + \mathbb{E}[2\gamma_t\ip{\hat{u}_t}{x_\star-z_t}] + \mathbb{E}[\gamma_t^2\norm{u_t}^2]. \label{eqn:adaptos-stochastic-feasibility-a}
\end{align}
Since we assume $f$ is convex and the solution lies in the interior of the feasible set, we know $0 \in \partial f(x_\star)$ and $\ip{\hat{u}_t}{x_\star-z_t} \leq 0$. Hence, we have 
\begin{align}
\mathbb{E}[\norm{y_{t+1} - x_\star}^2]
& \leq \mathbb{E}[\norm{y_t-x_\star}^2] + \mathbb{E}[\gamma_t^2\norm{u_t}^2] 
\leq \norm{y_0-x_\star}^2 + \mathbb{E}\big[\sum_{\tau=0}^t \gamma_\tau^2\norm{u_\tau}^2 \big].
\end{align}
By using \Cref{cor:adaptive-bounds}, this leads to
\begin{align}
\mathbb{E}[\norm{y_{t+1} - x_\star}^2] 
& \leq \norm{y_0 - x_\star}^2 + \alpha^2 \left( 2 + \tfrac{4G_f^2}{\beta} + 2\log\big(1 + \tfrac{G_f^2}{\beta}t \big) \right).
\end{align}
We take the square-root of both sides. Note that $\mathbb{E}[\norm{y_{t+1} - x_\star}]^2 \leq \mathbb{E}[\norm{y_{t+1} - x_\star}^2]$, hence
\begin{align}
\mathbb{E}[\norm{y_{t+1} - x_\star}] 
& \leq \norm{y_0 - x_\star} + \alpha \sqrt{ 2 + \tfrac{4G_f^2}{\beta} + 2\log\big(1 + \tfrac{G_f^2}{\beta}t \big) }.
\end{align}
Similar to \eqref{eqn:adaptos-nonsmooth-a-3}, we can use this bound to prove convergence to a feasible point:
\begin{align}
\mathbb{E}[\dist(\bar{z}_t,\mathcal{H})]
 & \leq \frac{1}{t+1} \Big(\mathbb{E}[\norm{y_{t+1} - x_\star}] + \norm{y_0 - x_\star} \Big)  \notag \\
 & \leq \frac{1}{t+1} \Big(2\norm{y_0 - x_\star} + \alpha \sqrt{ 2 + \tfrac{4G_f^2}{\beta} + 2\log\big(1 + \tfrac{G_f^2}{\beta}t \big)} \Big).
\end{align}

Next, we analyze convergence in the function value. Note that
\begin{align}
\gamma_\tau (f(z_\tau) - f_\star)
\leq \gamma_\tau \ip{\hat{u}_t}{z_\tau - x_\star}  
= \gamma_\tau \ip{u_\tau}{z_\tau - x_\star} + \gamma_\tau \ip{\hat{u}_\tau - u_\tau}{z_\tau - x_\star}.
\end{align}
Recall that $\gamma_\tau$ and $u_\tau$ are independent given $z_\tau$. Then, the second term vanishes if we take the expectation of both sides:
\begin{align}
\mathbb{E}[\gamma_\tau (f(z_\tau) - f_\star)] 
= \mathbb{E}[\gamma_\tau \ip{u_\tau}{z_\tau - x_\star}].
\end{align}
Then, by using \eqref{eqn:adaptos-stochastic-feasibility-a}, we get 
\begin{align}
\mathbb{E}[\gamma_\tau (f(z_\tau) - f_\star)] 
\leq \frac{1}{2} \mathbb{E}\big[ \norm{y_\tau - x_\star}^2 - \norm{y_{\tau+1} - x_\star}^2 + \gamma_\tau^2 \norm{u_\tau}^2 \big].
\end{align}
If we sum this inequality over $\tau = 0, 1, \ldots , t$, we get
\begin{align}
\mathbb{E}\Big[\sum_{\tau=0}^t \gamma_\tau (f(z_\tau) - f_\star)\Big] 
\leq \frac{1}{2} \norm{y_0 - x_\star}^2 + \frac{1}{2} \mathbb{E}\Big[\sum_{\tau=0}^t \gamma_\tau^2 \norm{u_\tau}^2 \Big]. 
\label{eqn:adaptops-stochastic-nonsmooth-a}
\end{align}
From \Cref{cor:adaptive-bounds},  
\begin{align}
\mathbb{E}\Big[\sum_{\tau=0}^t \gamma_\tau^2 \norm{u_\tau}^2 \Big]
\leq \alpha^2 \Big(2 + \tfrac{4G_f^2}{\beta} + 2\log\big(1 + \tfrac{G_f^2}{\beta}t \big)\Big).
\end{align}
Replacing this back into \eqref{eqn:adaptops-stochastic-nonsmooth-a}, we get
\begin{align}
\mathbb{E}\Big[\sum_{\tau=0}^t \gamma_\tau (f(z_\tau) - f_\star)\Big] 
\leq \frac{1}{2} \norm{y_0 - x_\star}^2 + \alpha^2 \Big(1 + \tfrac{2G_f^2}{\beta} + \log\big(1 + \tfrac{G_f^2}{\beta}t \big)\Big). 
\label{eqn:adaptops-stochastic-nonsmooth-b}
\end{align}
Let us define $s_t := \sum_{\tau=0}^t \gamma_\tau$. By Jensen's inequality, we get
\begin{align}
\mathbb{E}\Big[ \sum_{\tau=0}^t \gamma_\tau (f(z_\tau) - f_\star)\Big] 
= \mathbb{E}\Big[ \frac{s_t}{s_t} \sum_{\tau=0}^t \gamma_\tau (f(z_\tau) - f_\star)\Big] 
\geq \mathbb{E}\big[ s_t (f(\tilde{z}_t) - f_\star)\big].
\label{eqn:adaptops-stochastic-nonsmooth-c}
\end{align}
Note that 
\begin{align}
s_t = \sum_{\tau=0}^t \gamma_\tau
= \sum_{\tau=0}^t \frac{\alpha}{\sqrt{\beta + \sum_{j=0}^{\tau-1} \norm{u_j}^2}} 
\geq \sum_{\tau=0}^t \frac{\alpha}{ \sqrt{\beta + G_f^2 t}} 
= \frac{\alpha (t+1)}{\sqrt{\beta + G_f^2 t}} .
\label{eqn:adaptops-stochastic-nonsmooth-d}
\end{align}
Then, we have
\begin{align}
\frac{\alpha (t+1)}{\sqrt{\beta + G_f^2 t}} \mathbb{E}\big[(f(\tilde{z}_t) - f_\star)\big]
\leq \mathbb{E}\big[ s_t (f(\tilde{z}_t) - f_\star)\big] 
\leq \frac{1}{2} \norm{y_0 - x_\star}^2 + \alpha^2 \Big(1 + \tfrac{2G_f^2}{\beta} + \log\big(1 + \tfrac{G_f^2}{\beta}t \big)\Big).
\end{align}
By rearranging, we get
\begin{align}
\mathbb{E}\big[(f(\tilde{z}_t) - f_\star)\big]
\leq \Big(\frac{\sqrt{\beta}}{t+1} + \frac{G_f }{\sqrt{t+1}} \Big)  \bigg( \frac{1}{2\alpha} \norm{y_0 - x_\star}^2 + \alpha \Big(1 + \tfrac{2G_f^2}{\beta} + \log\big(1 + \tfrac{G_f^2}{\beta}t \big)\Big) \bigg).
\end{align}

\section{More Details on the Experiments in \texorpdfstring{\Cref{sec:numerical-experiments}}{Section~\ref{sec:numerical-experiments}}}

\subsection{Details for \texorpdfstring{\Cref{sec:experiments-smooth}}{Section~\ref{sec:experiments-smooth}}}

In the implementation of \algo we simply discarded $\beta$ and set $\gamma_0 = \alpha$. \Cref{fig:tuning} demonstrates how the performance of \algo depends on $\alpha$ for the experiments we considered in \Cref{sec:experiments-smooth}. 

For \Cref{fig:smooth}, we choose: \\[0.5em]
$~\quad\triangleright$ $\alpha = 10$ for overlapping group lasso with $\lambda = 10^{-3}$ and synthetic data, \\[0.5em]
$~\quad\triangleright$ $\alpha = 1$ for overlapping group lasso with $\lambda = 10^{-1}$ and synthetic data, \\[0.5em]
$~\quad\triangleright$ $\alpha = 100$ for overlapping group lasso with $\lambda = 10^{-3}$ and real-sim dataset, \\[0.5em]
$~\quad\triangleright$ $\alpha = 100$ for overlapping group lasso with $\lambda = 10^{-2}$ and real-sim dataset, \\[0.5em]
$~\quad\triangleright$ $\alpha = 1$ for sparse and low-rank regularization with $\lambda = 10^{-3}$, \\[0.5em]
$~\quad\triangleright$ $\alpha = 1$ for sparse and low-rank regularization with $\lambda = 10^{-4}$, \\[0.5em]
$~\quad\triangleright$ $\alpha = 100$ for total variation deblurring with $\lambda = 10^{-6}$, \\[0.5em]
$~\quad\triangleright$ $\alpha = 100$ for total variation deblurring with $\lambda = 10^{-4}$. 

In \Cref{sec:experiments-smooth}, we consider problems only with smooth $f$. To present how the performance of \algo changes by $\alpha$ when $f$ is nonsmooth, we also run the overlapping group lasso problem with the hinge loss. In this setting, we used RCV1 dataset \citep{lewis2004rcv1} ($n = 677399$, $N = 20242$) and tried two different values of the regularization parameter $\lambda = 10^{-3}$ and $10^{-2}$. The results are shown in \Cref{fig:og-hinge-loss}. \vspace{0.5em}

\begin{figure}[h]
\begin{center}
\includegraphics[width=\linewidth]{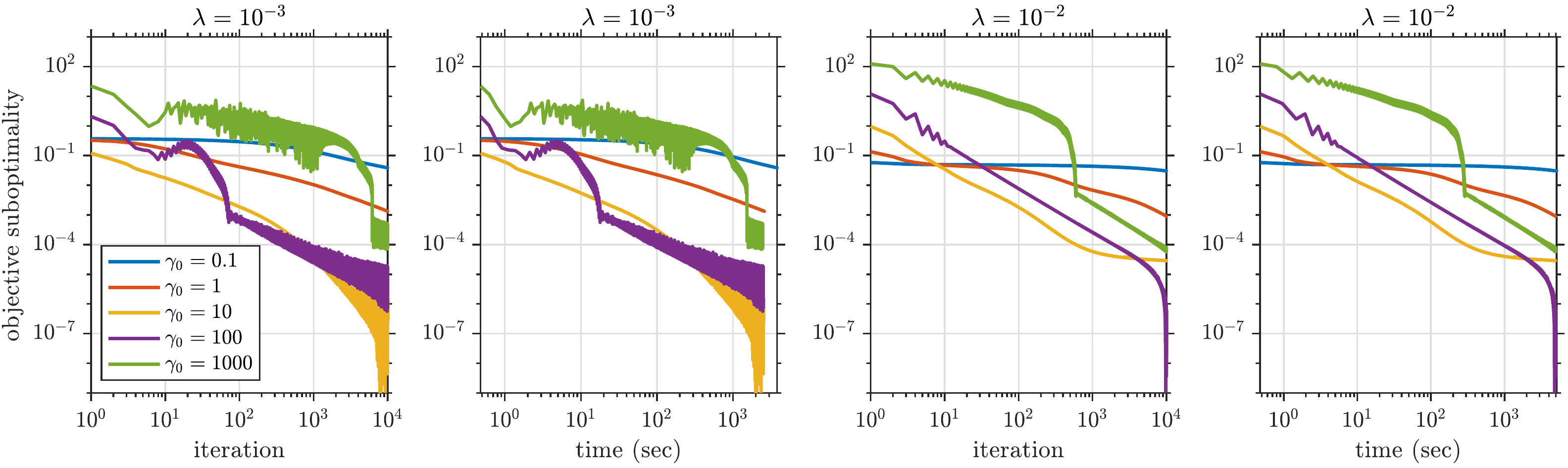} \\[0.5em]
\caption{Empirical performance of \algo for different choices of $\alpha$ for the overlapping group lasso problem with the hinge-loss. In this experiment we use RCV1 dataset \citep{lewis2004rcv1} ($n = 677399$, $N = 20242$) and tried two different values of the regularization parameter $\lambda = 10^{-3}$ and $10^{-2}$.}
\label{fig:og-hinge-loss}
\vspace{1em}
\end{center}
\end{figure}

\begin{figure}[p]
\begin{center}
\includegraphics[width=\linewidth]{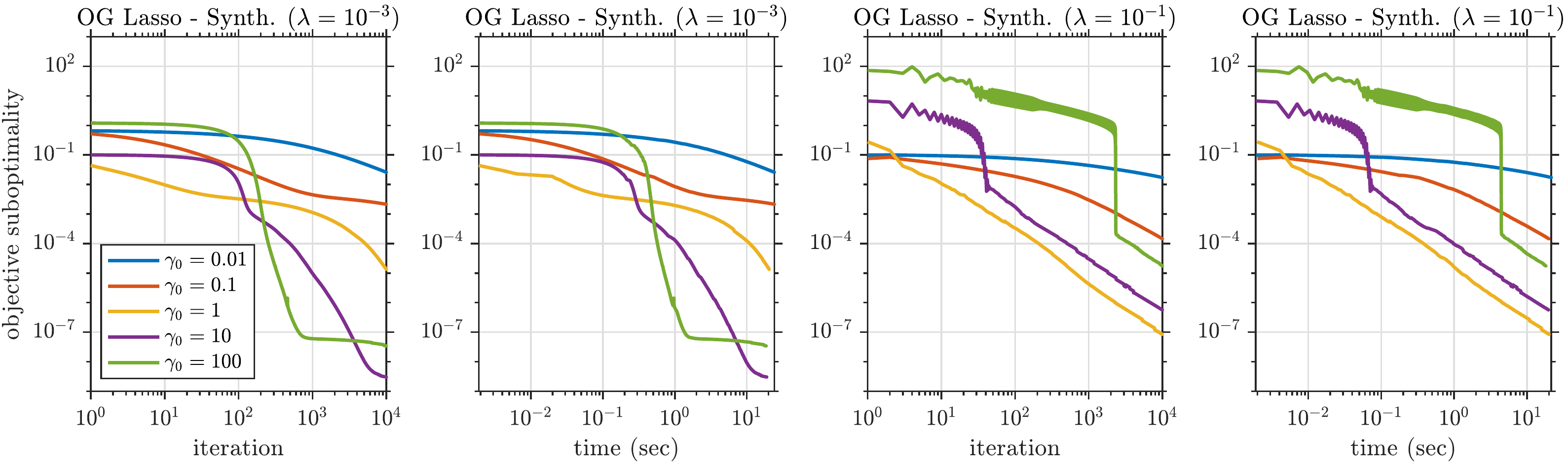} \\[0.5em]
\includegraphics[width=\linewidth]{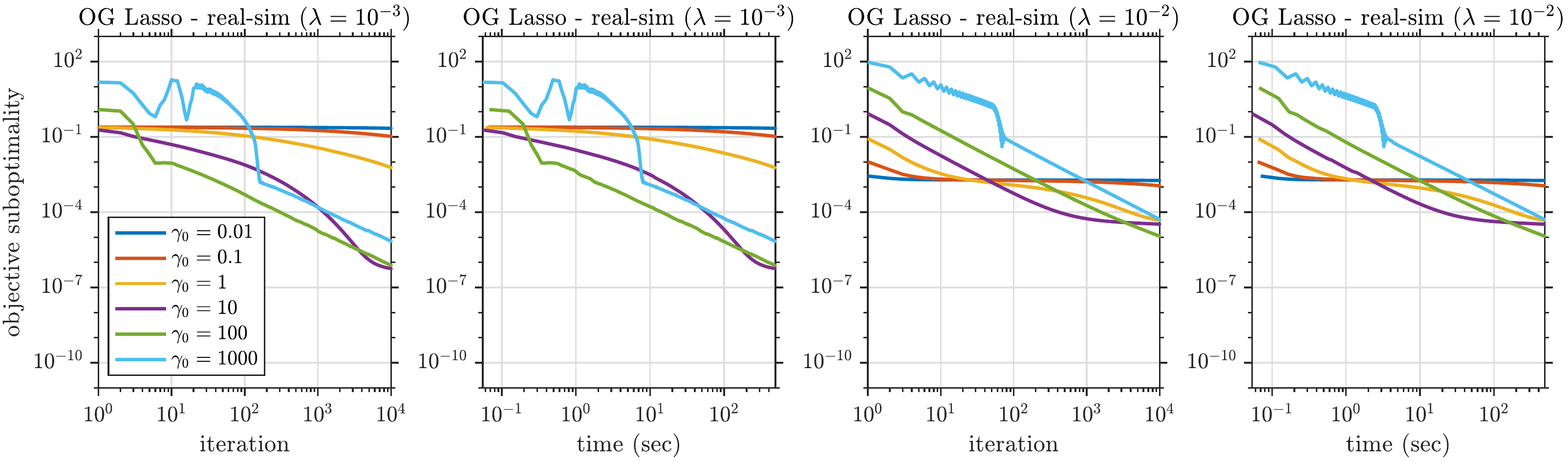} \\[0.5em]
\includegraphics[width=\linewidth]{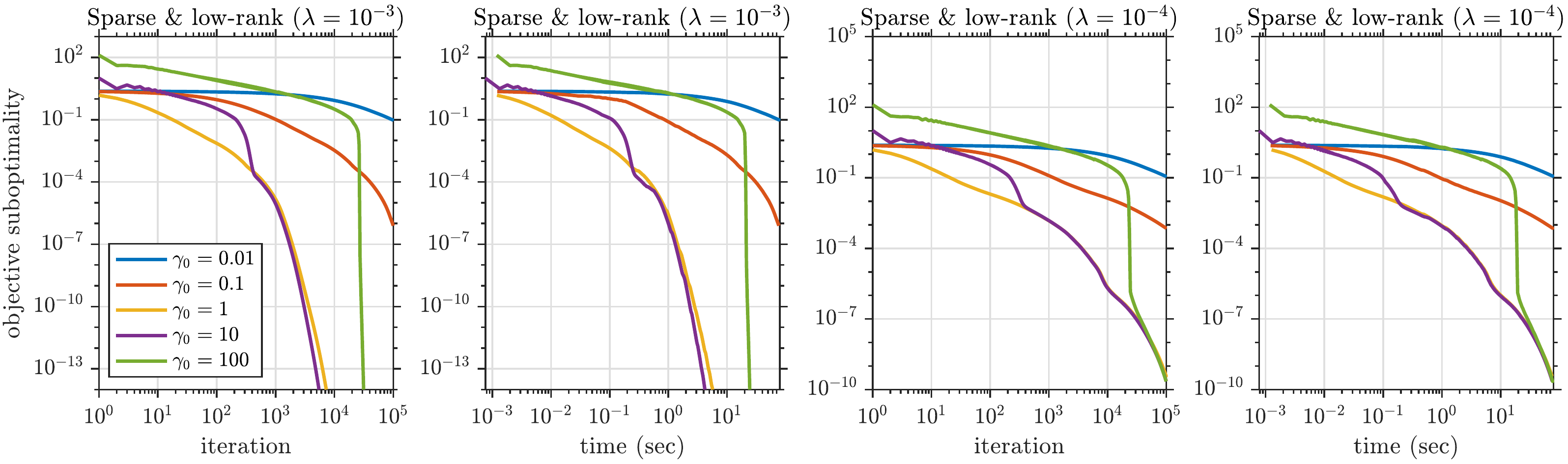} \\[0.5em]
\includegraphics[width=\linewidth]{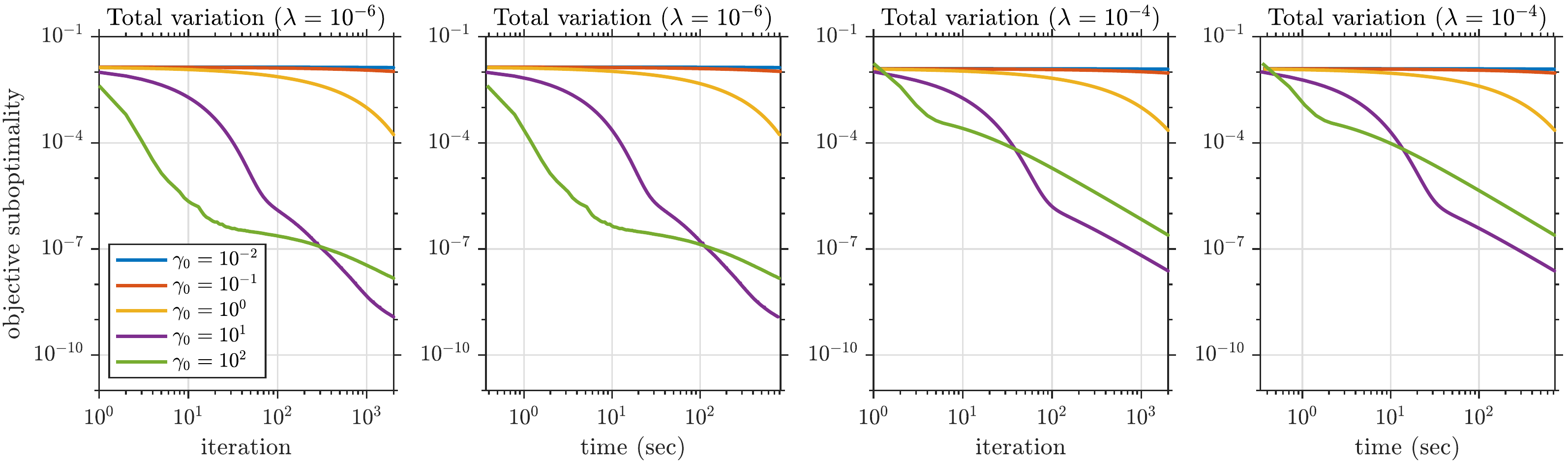} 
\caption{Empirical performance of \algo with different choices of $\alpha$ for the problems with smooth and convex loss function studied in \Cref{sec:experiments-smooth}.}
\label{fig:tuning}
\vspace{1em}
\end{center}
\end{figure}

\subsection{Details for \texorpdfstring{\Cref{sec:experiments-nonsmooth}}{Section~\ref{sec:experiments-nonsmooth}}}

\Cref{fig:recovered-images} shows the recovered approximations with $\ell_1$ and $\ell_2$-loss functions along with the original image and the noisy observation. $\ell_1$-loss is known to be more reliable against outliers, and it empirically generates a better approximation of the original image with 26.21 dB peak signal to noise ratio (PSNR) against 21.15 dB for the $\ell_2$-loss. 

In \Cref{fig:image-recovery-supp} we extend the comparison in \Cref{fig:nonsmooth} with the squared-$\ell_2$ loss. 
\begin{equation}
\min_{\boldsymbol{X} \in \R^{m \times n}} \quad \frac{1}{2}\norm{\mathcal{A}(X) - Y}_2^2 \quad \text{subject to} \quad \norm{X}_\ast \leq \lambda, ~~~ 0\leq X \leq 1,
\end{equation}
Note that the solution set is the same for $\ell_2$ and squared-$\ell_2$ formulations. However, squared-$\ell_2$ loss is smooth whereas $\ell_2$ loss is nonsmooth. Nevertheless, the empirical performance of \algo for the two formulations are similar. We also compare the evaluation of PSNR over the iterations. This comparison clearly demonstrates the advantage of using the robust $\ell_1$ loss formulation. 

\begin{figure}[p]
\begin{center}
\includegraphics[scale=0.5]{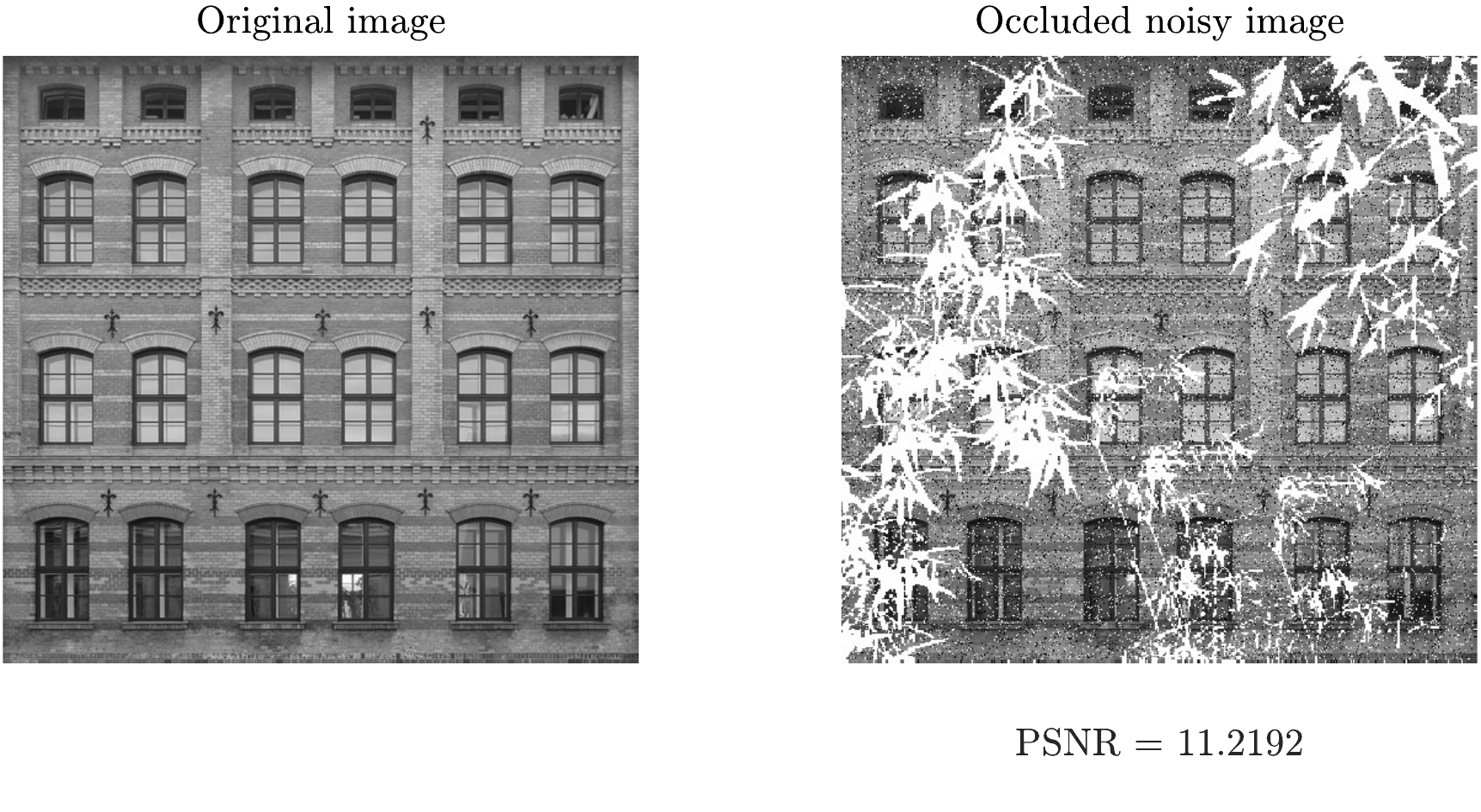}
\includegraphics[scale=0.5]{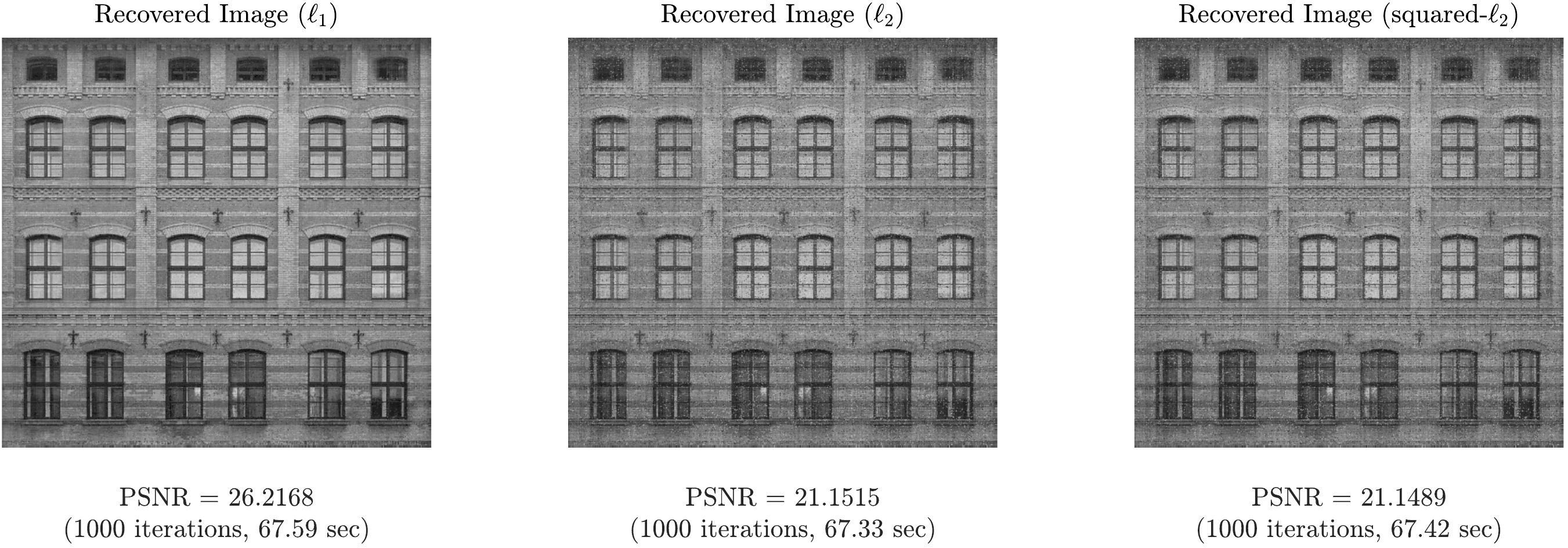}\vspace{-0.5em}
\caption{Comparison or images recovered by minimizing the $\ell_1$, $\ell_2$ and squared-$\ell_2$ loss functions described in \Cref{sec:experiments-nonsmooth}. $\ell_1$-loss empirically gives a better approximation with 5dB higher PSNR.} 
\label{fig:recovered-images}
\vspace{1.5em}
\end{center}

\begin{center}
\includegraphics[width=\linewidth]{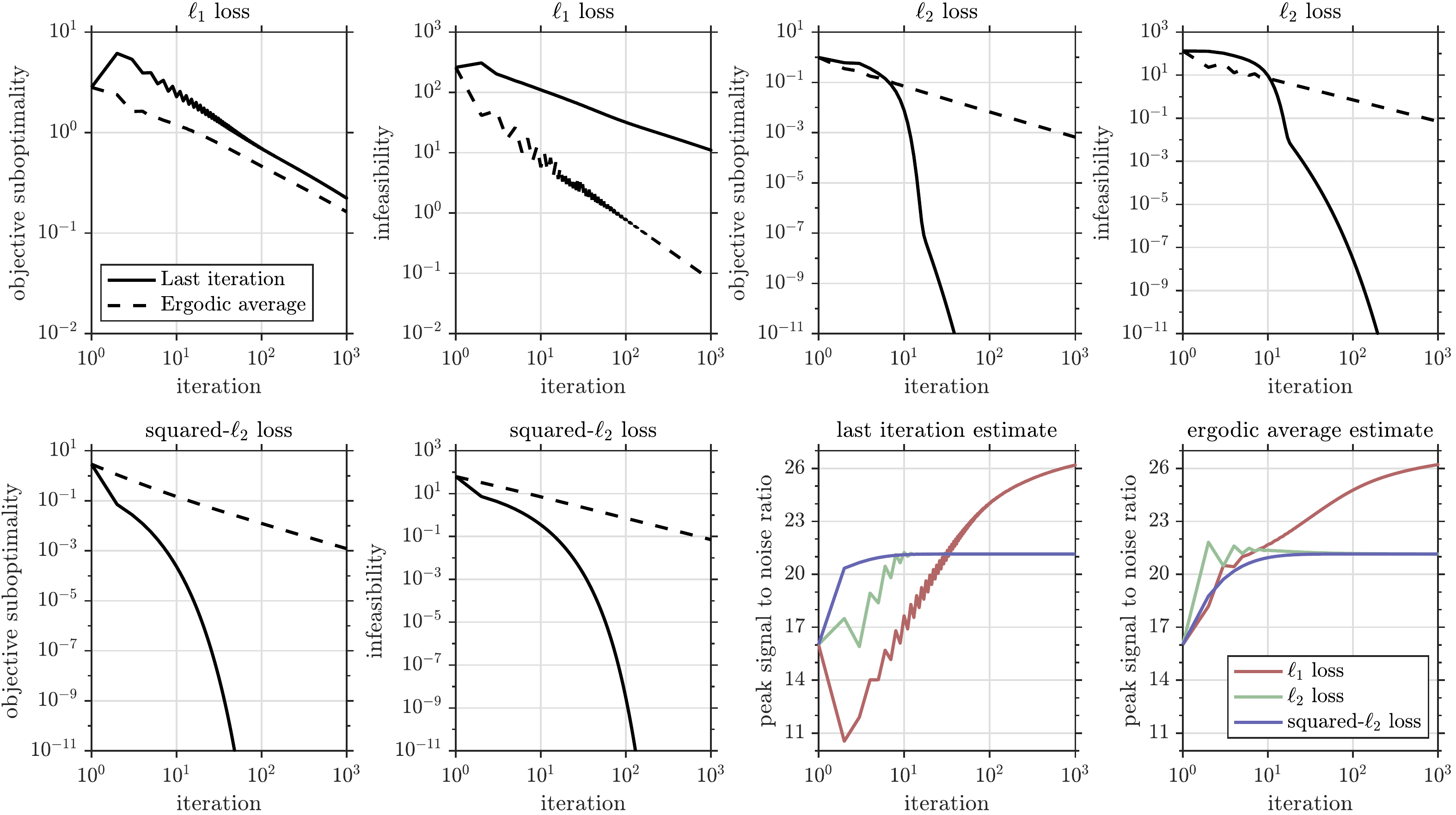}\vspace{-1em}
\caption{Performance of \algo on image impainting problems with $\ell_1$, $\ell_2$, and squared-$\ell_2$ loss functions described in \Cref{sec:experiments-nonsmooth}. The performance for nonsmooth $\ell_2$ loss and smooth squared-$\ell_2$ loss are qualitatively similar.}
\label{fig:image-recovery-supp}
\end{center}
\end{figure}

\subsection{Details for \texorpdfstring{\Cref{sec:experiments-mnist}}{Section~\ref{sec:experiments-mnist}}}

Let $\boldsymbol{y} = f(\boldsymbol{x},\boldsymbol{w})$ denote a generic deep neural network with $H$ hidden layers, which takes a vector input $x$ and returns a vector output $y$, and $w$ represents the column-vector concatenation of all adaptable parameters. $k$th hidden layer operates on input vector $\boldsymbol\theta_k$ and returns $\boldsymbol\theta_{k+1}$, 
\begin{align}
    \boldsymbol\theta_{k+1} & = \phi_k (\boldsymbol{W}_k \boldsymbol\theta_k + \boldsymbol{b}_k), \quad \text{for}~1\leq k \leq H, 
\end{align}
where $\boldsymbol\theta_1 = \boldsymbol{x}$ denotes the input layer by convention, $\{\boldsymbol\theta_k,\boldsymbol{b}_k\}$ are the adaptable parameters of the layer and $\phi_k$ is an activation function to be applied entry-wise. We use ReLu activation \citep{glorot2011deep} for the hidden layers of the network and the softmax activation function for the output layer. We use the same initial weights as in \citep{scardapane2017group}, which is based on the method described in \citep{glorot2010understanding}. 

Given a set of $N$ training examples $\{(\boldsymbol{x}_1,\boldsymbol{y}_1), \ldots, (\boldsymbol{x}_N,\boldsymbol{y}_N)\}$, we train the network by minimizing 
\begin{equation}
\label{eqn:neural-network}
\min_{\boldsymbol{w} \in \R^n} \quad \frac{1}{N} \sum_{i = 1}^N L(\boldsymbol{y}_i, f(\boldsymbol{x}_i,\boldsymbol{w})) + \lambda \norm{\boldsymbol{w}}_1 + \lambda \sum_{\boldsymbol\alpha \in \Omega} \sqrt{|\boldsymbol\alpha|} \norm{\boldsymbol\alpha}_2,
\end{equation}
with the standard cross-entropy loss given by $L(\boldsymbol{y},f(\boldsymbol{x},\boldsymbol{w})) = -\sum_{j=1}^{\mathrm{dim}(\boldsymbol{y})} y_j \log(f_j(\boldsymbol{x}, \boldsymbol{w}))$. $\lambda > 0$ is the regularization parameter. We set $\lambda = 10^{-4}$, which is shown to provide the best results in terms of classification accuracy and sparsity in \citep{scardapane2017group}. 

The first regularizer ($\ell_1$ penalty) in \eqref{eqn:neural-network} promotes sparsity on the overall network, while the second regularizer (Group-Lasso penalty, introduced in \citep{yuan2006model}) is used to achieve group-level sparsity. The goal is to force all outgoing connections from the same neurons to be simultaneously zero, so that we can safely remove them and obtain a compact network. To this end, $\Omega$ contains the sets of all outgoing connections from each neuron (corresponding to the rows of $\boldsymbol{W}_k$) and single element groups of bias terms (corresponding to the entries of $\boldsymbol{b}_k$). 

We compare our methods against SGD, AdaGrad and Adam. We use minibatch size of $400$ for all methods. We use the built-in functions in Lasagne for SGD, AdaGrad and Adam. These methods use the subgradient of the overall objective \eqref{eqn:neural-network}. All of these methods have one learning rate parameter for tuning. We tune these parameters by trying the powers of $10$. We found that $\gamma_0 = \alpha = 1$ works well for TOS and \algo. For SGD and AdaGrad, we got the best performance when the learning rate parameter is set to $10^{-2}$, and for Adam we got the best results with $10^{-3}$. %

Remark that subgradient methods are known to destroy sparsity at the intermediate iterations. For instance, the subgradients of $\ell_1$ norm are fully dense. In contrast, TOS and \algo handle the regularizers through their proximal operators. The advantage of using a proximal method instead of subgradients is outstanding. TOS and \algo result in precisely sparse networks whereas other methods can only get approximately sparse solutions. The comparison becomes especially stark in group sparsity, with no clear discontinuity in the spectrum for other methods.

\section{Additional Numerical Experiments}

In this section, we present additional numerical experiments on isotonic regression and portfolio optimization problems. The experiments in this section are performed in \textsc{Matlab} R2018a with 2.6 GHz Quad-Core Intel Core i7 CPU. 

\subsection{Isotonic Regression}

In this section, we compare the empirical performance of the adaptive step-size in \Cref{sec:adaptos} with the analytical step-size in \Cref{sec:algorithm-nonsmooth}. We consider the isotonic regression problem with the $\ell_p$-norm loss:
\begin{align}
\label{eqn:supp-isotonic-regression}
\min_{x \in \R^n} \quad \tfrac{1}{p} \norm{Ax - b}_p^p \quad \mathrm{subject~to} \quad x_1 \leq x_2 \leq \ldots \leq x_n,
\end{align}
where $A \in \mathbb{R}^{m \times n}$ is a given linear map and $b \in \mathbb{R}^m$ is the measurement vector. 
Projection onto the order constraint in \eqref{eqn:supp-isotonic-regression} is challenging, but we can split it into two simpler constraints: 
\begin{align}
\label{eqn:supp-isotonic-regression-split}
\min_{x \in \R^n} \quad \tfrac{1}{p} \norm{Ax - b}_p^p \quad \mathrm{subject~to} \quad \left\{\begin{aligned} x_1 \leq x_2\\
x_3 \leq x_4 \\[-0.5em]
\vdots\phantom{~~x_1}
\end{aligned}\right\}
~~~\mathrm{and}~~~ 
\left\{\begin{aligned} x_2 \leq x_3\\
x_4 \leq x_5 \\[-0.5em]
\vdots\phantom{~~x_1}
\end{aligned}\right\}.
\end{align}

We demonstrate the numerical performance of the methods for various values of $p \in [1,2]$. 
Note that $p=1$ and $p=2$ capture the nonsmooth least absolute deviations loss and the smooth least squares loss respectively. For larger values of $p$, we expect \algo to exhibit faster rates by adapting to the underlying smoothness of the objective function. 

We generate a synthetic test setup. To this end, we set the problem size as $m=100$ and $n=200$. We generate right and left singular vectors of $A$ by computing the singular value decomposition of a random matrix with \textit{iid} entries drawn from the standard normal distribution. Then, we set the singular values according to a polynomial decay rule such that the $i$th singular vector is $1/i$. We generate $\smash{x^\natural \in \mathbb{R}^n}$ by sorting $n$ \textit{iid} samples from the standard normal distribution. Then, we compute the noisy measurements $\smash{b = Ax^\natural + 0.1 \xi}$ where the entries of $\xi$ is drawn \textit{iid} from the standard normal distribution. 

By considering a decaying singular value spectrum for $A$ we control the condition number and make sure the problem is not very easy to solve.  By adding noise, we ensure that the solution is not in the relative interior of the feasible set. Therefore, this experiment also supports our claim that \algo can achieve fast rates when the objective is smooth even if the solution does not lie in the interior of the feasible set. 

When the problem is nonsmooth, \textit{i.e.}, when $p < 2$, we use TOS with the analytical step-size in \Cref{sec:algorithm-nonsmooth} and the adaptive step-size in \Cref{sec:adaptos}. We choose $\alpha = \beta = \gamma_0 = 1$ without any tuning. When $p=2$, the problem is smooth so we also try TOS with the standard constant step-size $\gamma = 1/L_f$ in this setting. We run each algorithm for $10^5$ iterations. In order to find the ground truth $f_\star$ we solve the problem to very high precision by using CVX \citep{grant2014cvx} with the SDPT3 solver \citep{toh1999sdpt3}. 

We repeat the experiments with $20$ randomly generated data with different seeds and report the average performance in \Cref{fig:isotonic-regression}. This figure compares the performance we get by different step-size strategies in terms of objective suboptimality ($|f(z_t) - f_\star|/f_\star$) and infeasbility bound ($\norm{z_t - x_t}$). As expected, \algo performs better as $p$ becomes larger. Although it does not exactly match the performance of the fixed step-size $1/L_f$ when $f$ is smooth ($p=2$), remark that \algo does not require any prior knowledge on $L_f$ or $G_f$. %

\begin{figure}[t!]
\begin{center}
\includegraphics[width=\linewidth]{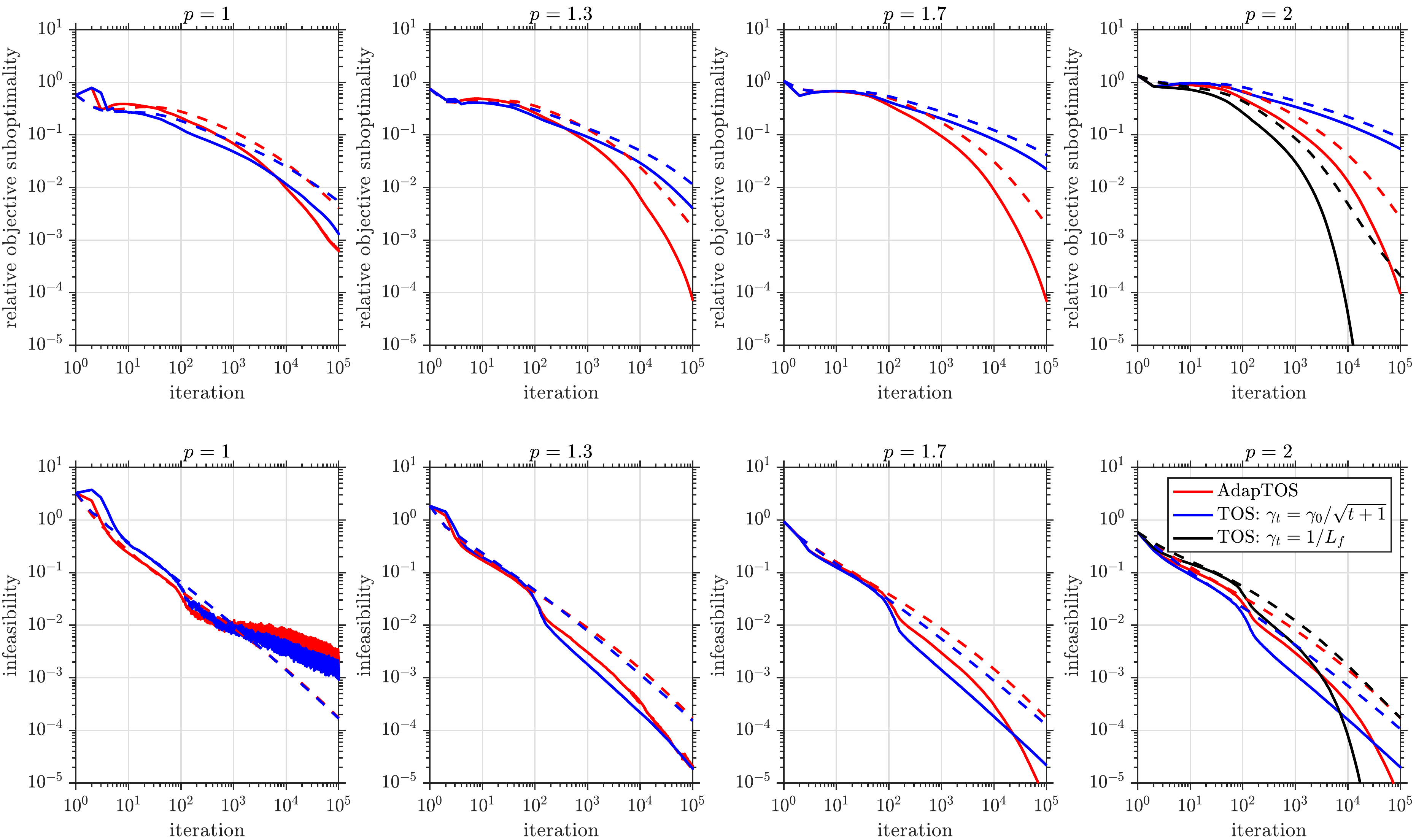}\vspace{-1em}
\caption{Comparison of the empirical performance of TOS with the analytical step-size in \Cref{sec:algorithm-nonsmooth} and the adaptive step-size in \Cref{sec:adaptos} on the isotonic regression problem in \eqref{eqn:supp-isotonic-regression} with the $\ell_p$-loss function for various $p \in [1,2]$. For larger values of $p$, \algo exhibits faster convergence rates by adapting to the underlying smoothness of the objective function. For $p=2$, the problem is smooth so we also consider the standard fixed step-size $\gamma = 1/L_f$ in this setting. Solid lines represent the last iteration and the dashed lines correspond to the ergodic sequences $\bar{x}_t$ and $\bar{z}_t$.}
\label{fig:isotonic-regression}
\end{center}
\end{figure}

\subsection{Portfolio Optimization}

In this section, we demonstrate the advantage of stochastic methods for machine learning problems. We consider the portfolio optimization with empirical risk minimization from Section~5.1 in \citep{yurtsever2016stochastic}:
\begin{align}
\label{eqn:supp-portfolio-optimization}
\min_{x \in \R^n} \quad \frac{1}{2}\sum_{i=1}^N |\ip{a_i}{x} - b|^2 \quad \mathrm{subject~to} \quad x \in \Delta \quad \text{and} \quad \ip{a_{\mathrm{av}}}{x} \geq b
\end{align}
where $\Delta$ is the unit simplex. Here $n$ is the number of different assets and $x \in \Delta$ represents a portfolio. The collection of $\{a_i\}_{i=1}^N$ represents the returns of each asset at different time instances, and the $a_\mathrm{av}$ is the average returns for each asset that is assumed to be known or estimated. Given a minimum target return $b \in \mathbb{R}$, the goal is to reduce the risk by minimizing the variance. As in \citep{yurtsever2016stochastic}, we set the target return as the average return over all assets, \textit{i.e.}, $b = \mathrm{mean}(a_{\mathrm{av}})$. 

In addition, we also consider a modification of \eqref{eqn:supp-portfolio-optimization} with the least absolute deviation loss, which is nonsmooth but known to be more robust against outliers:
\begin{align}
\label{eqn:supp-portfolio-optimization-nonsmooth}
\min_{x \in \R^n} \quad \sum_{i=1}^N |\ip{a_i}{x} - b| \quad \mathrm{subject~to} \quad x \in \Delta \quad \text{and} \quad \ip{a_{\mathrm{av}}}{x} \geq b.
\end{align}

We use $4$ different real portfolio datasets: Dow Jones industrial average (DJIA, 30 stocks for
507 days), New York stock exchange (NYSE, 36 stocks for 5651 days), Standard \& Poor's 500
(SP500, 25 stocks for 1276 days), and Toronto stock exchange (TSE, 88 stocks for 1258 days).\footnote{These four datasets can be downloaded from \href{http://www.cs.technion.ac.il/~rani/portfolios/}{http://www.cs.technion.ac.il/{\textasciitilde}rani/portfolios/}}

For both problems and each dataset, we run \algo with full (sub)gradients and stochastic (sub)gradients and compare their performances. We choose $\alpha = \beta = 1$ without tuning and run the algorithms for $10$ epochs. In the stochastic setting, we evaluate a (sub)gradient estimator from a single datapoint chosen uniformly at random with replacement at every iteration. We run the stochastic algorithm 20 times with different random seeds and present the average performance. To find the ground truth $f_\star$ we solve the problems to very high precision by using CVX \citep{grant2014cvx} with the SDPT3 solver \citep{toh1999sdpt3}. 
\Cref{fig:portfolio-optimization-smooth,fig:portfolio-optimization-nonsmooth} present the results of this experiment for \eqref{eqn:supp-portfolio-optimization} and \eqref{eqn:supp-portfolio-optimization-nonsmooth} respectively. 

\begin{figure}[p]
\begin{center}
\includegraphics[width=\linewidth]{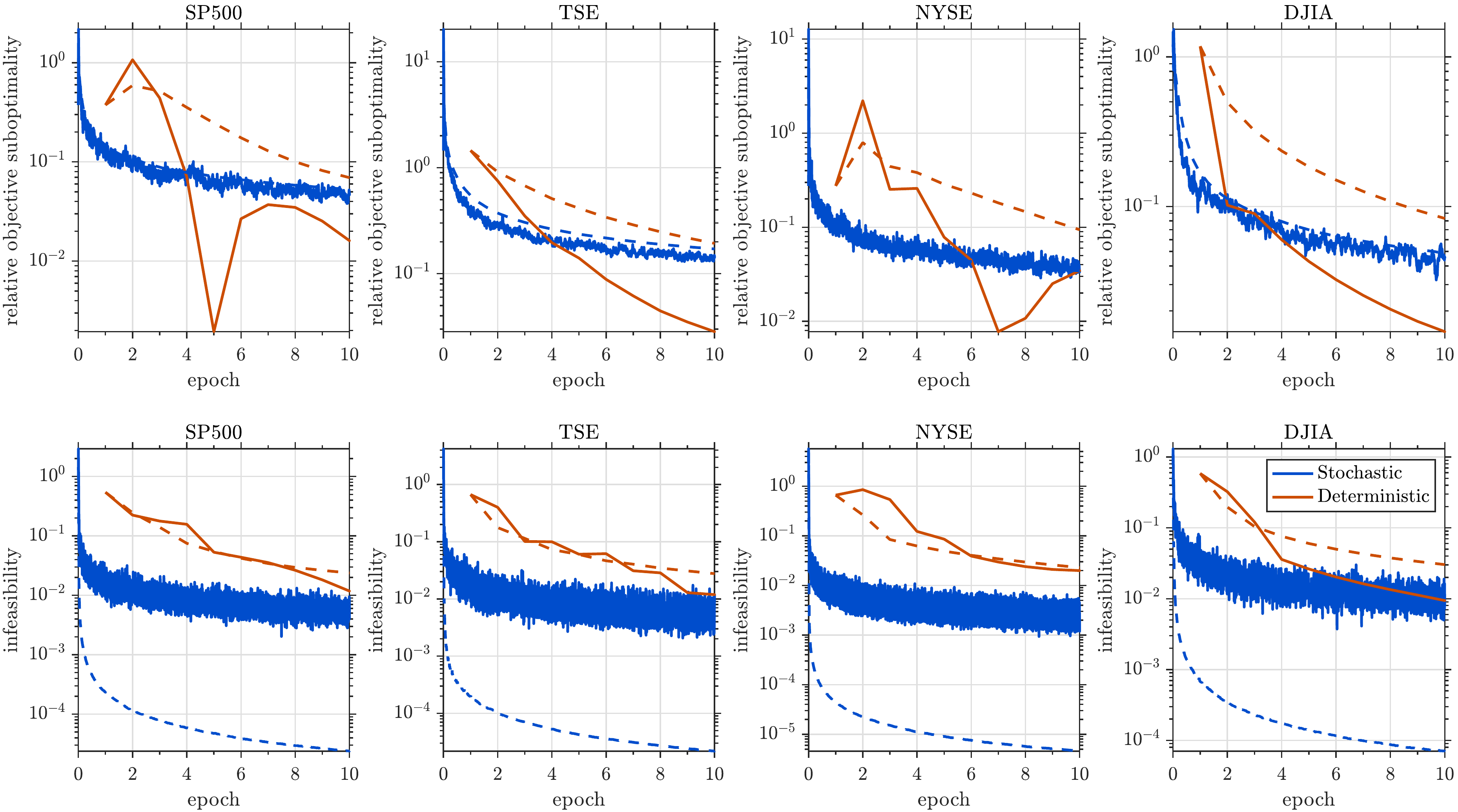}
\vspace{-1.5em}
\caption{Comparison of \algo with stochastic and deterministic gradients on the smooth portfolio optimization problem with least-squares loss in \eqref{eqn:supp-portfolio-optimization} for four different datasets. Solid and dashed lines represent the last and ergodic iterates respectively.}
\label{fig:portfolio-optimization-smooth}
\vspace{1.5em}

\includegraphics[width=\linewidth]{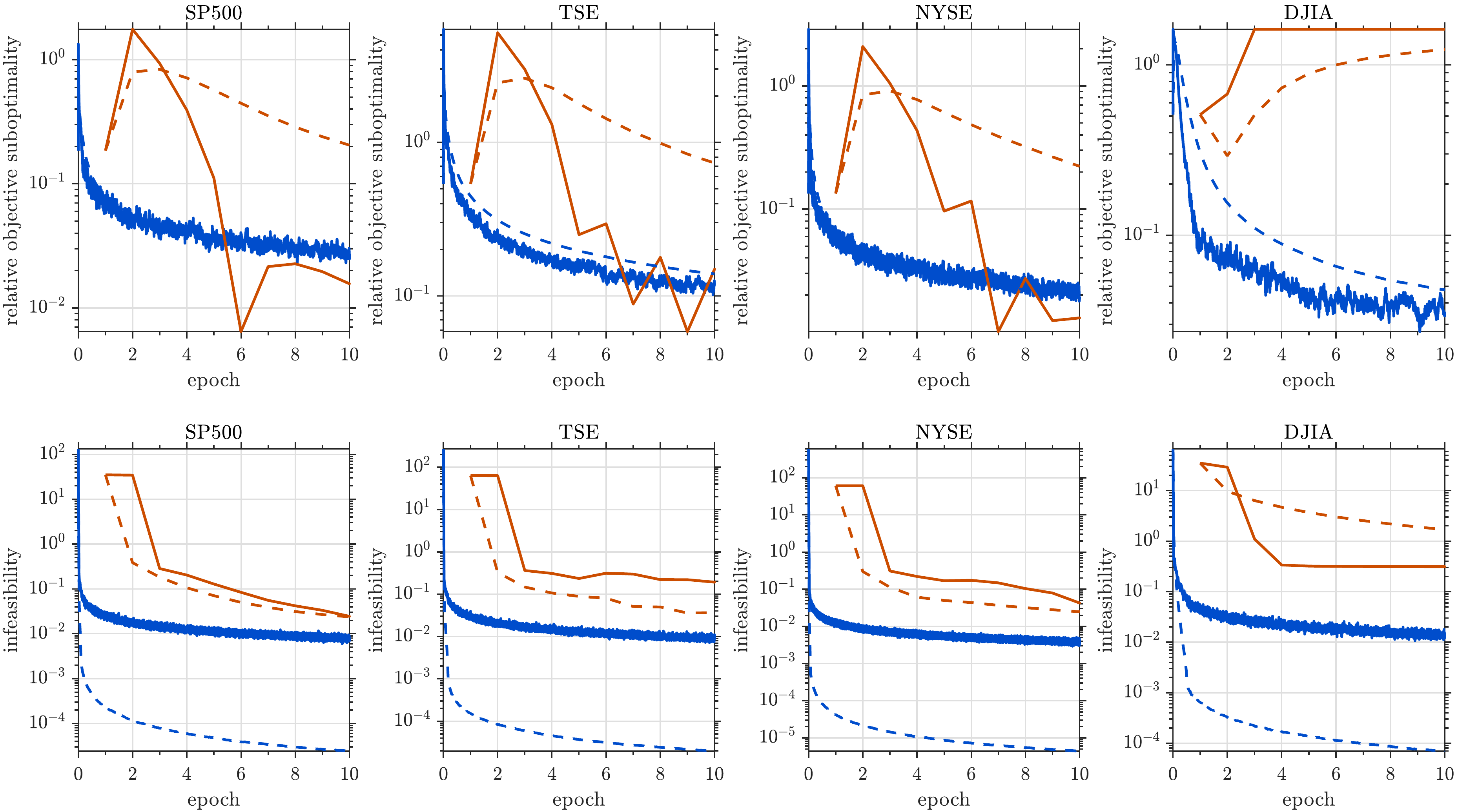}\vspace{-1em}
\caption{Comparison \algo with stochastic and deterministic gradients on the nonsmooth portfolio optimization problem with least absolute deviations loss in \eqref{eqn:supp-portfolio-optimization-nonsmooth} for four different datasets. Solid and dashed lines represent the last and ergodic iterates respectively.}
\label{fig:portfolio-optimization-nonsmooth}
\end{center}
\end{figure}

\end{appendices}
 
\section*{Acknowledgements}
\label{sec:acknowledgements}

The authors would like to thank Ahmet Alacaoglu for carefully reading and reporting an error in the preprint of this paper. 

Alp Yurtsever received support from the Swiss National Science Foundation Early Postdoc.Mobility Fellowship P2ELP2\_187955, from the Wallenberg AI, Autonomous Systems and Software Program (WASP) funded by the Knut and Alice Wallenberg Foundation, and partial postdoctoral support from the NSF-CAREER grant IIS-1846088. Suvrit Sra acknowledges support from an NSF BIGDATA grant (1741341) and an NSF CAREER grant (1846088).

\bibliography{NonsmoothTOS}

\begin{thebibliography}{53}
\providecommand{\natexlab}[1]{#1}
\providecommand{\url}[1]{\texttt{#1}}
\expandafter\ifx\csname urlstyle\endcsname\relax
  \providecommand{\doi}[1]{doi: #1}\else
  \providecommand{\doi}{doi: \begingroup \urlstyle{rm}\Url}\fi

\bibitem[Azadi \& Sra(2014)Azadi and Sra]{azadi2014towards}
Azadi, S. and Sra, S.
\newblock Towards an optimal stochastic alternating direction method of
  multipliers.
\newblock In \emph{International Conference on Machine Learning}, pp.\
  620--628. PMLR, 2014.

\bibitem[Bach \& Levy(2019)Bach and Levy]{bach2019universal}
Bach, F. and Levy, K.~Y.
\newblock A universal algorithm for variational inequalities adaptive to
  smoothness and noise.
\newblock In \emph{Conference on Learning Theory}, pp.\  164--194. PMLR, 2019.

\bibitem[Barbero \& Sra(2018)Barbero and Sra]{barbero2018modular}
Barbero, A. and Sra, S.
\newblock Modular proximal optimization for multidimensional total-variation
  regularization.
\newblock \emph{The Journal of Machine Learning Research}, 19\penalty0
  (1):\penalty0 2232--2313, 2018.

\bibitem[Bauschke et~al.(2011)Bauschke, Combettes, et~al.]{bauschke2011convex}
Bauschke, H.~H., Combettes, P.~L., et~al.
\newblock \emph{Convex analysis and monotone operator theory in Hilbert
  spaces}, volume 408.
\newblock Springer, 2011.

\bibitem[Blalock et~al.(2020)Blalock, Ortiz, Frankle, and
  Guttag]{blalock2020state}
Blalock, D., Ortiz, J. J.~G., Frankle, J., and Guttag, J.
\newblock What is the state of neural network pruning?
\newblock \emph{arXiv preprint arXiv:2003.03033}, 2020.

\bibitem[Brice{\~n}o-Arias(2015)]{briceno2015forward}
Brice{\~n}o-Arias, L.~M.
\newblock Forward-douglas--rachford splitting and forward-partial inverse
  method for solving monotone inclusions.
\newblock \emph{Optimization}, 64\penalty0 (5):\penalty0 1239--1261, 2015.

\bibitem[Cevher et~al.(2018)Cevher, V{\~u}, and
  Yurtsever]{cevher2018stochastic}
Cevher, V., V{\~u}, B.~C., and Yurtsever, A.
\newblock Stochastic forward douglas-rachford splitting method for monotone
  inclusions.
\newblock In \emph{Large-Scale and Distributed Optimization}, pp.\  149--179.
  Springer, 2018.

\bibitem[Chambolle \& Pock(2016)Chambolle and Pock]{chambolle2016ergodic}
Chambolle, A. and Pock, T.
\newblock On the ergodic convergence rates of a first-order primal--dual
  algorithm.
\newblock \emph{Mathematical Programming}, 159\penalty0 (1):\penalty0 253--287,
  2016.

\bibitem[Chang \& Lin(2011)Chang and Lin]{chang2011libsvm}
Chang, C.-C. and Lin, C.-J.
\newblock Libsvm: A library for support vector machines.
\newblock \emph{ACM transactions on intelligent systems and technology (TIST)},
  2\penalty0 (3):\penalty0 1--27, 2011.

\bibitem[Condat(2013)]{condat2013primal}
Condat, L.
\newblock A primal--dual splitting method for convex optimization involving
  lipschitzian, proximable and linear composite terms.
\newblock \emph{Journal of optimization theory and applications}, 158\penalty0
  (2):\penalty0 460--479, 2013.

\bibitem[Cutkosky(2019)]{cutkosky2019anytime}
Cutkosky, A.
\newblock Anytime online-to-batch, optimism and acceleration.
\newblock In \emph{International Conference on Machine Learning}, pp.\
  1446--1454. PMLR, 2019.

\bibitem[Davis \& Yin(2017)Davis and Yin]{davis2017three}
Davis, D. and Yin, W.
\newblock A three-operator splitting scheme and its optimization applications.
\newblock \emph{Set-valued and variational analysis}, 25\penalty0 (4):\penalty0
  829--858, 2017.

\bibitem[Ding et~al.(2019)Ding, Yurtsever, Cevher, Tropp, and
  Udell]{ding2019optimal}
Ding, L., Yurtsever, A., Cevher, V., Tropp, J.~A., and Udell, M.
\newblock An optimal-storage approach to semidefinite programming using
  approximate complementarity.
\newblock \emph{arXiv preprint arXiv:1902.03373}, 2019.

\bibitem[Duchi et~al.(2011)Duchi, Hazan, and Singer]{duchi2011adaptive}
Duchi, J., Hazan, E., and Singer, Y.
\newblock Adaptive subgradient methods for online learning and stochastic
  optimization.
\newblock \emph{Journal of machine learning research}, 12\penalty0 (7), 2011.

\bibitem[El~Halabi \& Cevher(2015)El~Halabi and Cevher]{el2015totally}
El~Halabi, M. and Cevher, V.
\newblock A totally unimodular view of structured sparsity.
\newblock In \emph{Artificial Intelligence and Statistics}, pp.\  223--231.
  PMLR, 2015.

\bibitem[Glorot \& Bengio(2010)Glorot and Bengio]{glorot2010understanding}
Glorot, X. and Bengio, Y.
\newblock Understanding the difficulty of training deep feedforward neural
  networks.
\newblock In \emph{Proceedings of the thirteenth international conference on
  artificial intelligence and statistics}, pp.\  249--256. JMLR Workshop and
  Conference Proceedings, 2010.

\bibitem[Glorot et~al.(2011)Glorot, Bordes, and Bengio]{glorot2011deep}
Glorot, X., Bordes, A., and Bengio, Y.
\newblock Deep sparse rectifier neural networks.
\newblock In \emph{Proceedings of the fourteenth international conference on
  artificial intelligence and statistics}, pp.\  315--323. JMLR Workshop and
  Conference Proceedings, 2011.

\bibitem[Grant \& Boyd(2014)Grant and Boyd]{grant2014cvx}
Grant, M. and Boyd, S.
\newblock {CVX}: {M}atlab software for disciplined convex programming, version
  2.1, 2014.

\bibitem[He et~al.(2016)He, Zhang, Ren, and Sun]{he2016deep}
He, K., Zhang, X., Ren, S., and Sun, J.
\newblock Deep residual learning for image recognition.
\newblock In \emph{Proceedings of the IEEE conference on computer vision and
  pattern recognition}, pp.\  770--778, 2016.

\bibitem[Higham \& Strabi{\'c}(2016)Higham and Strabi{\'c}]{higham2016anderson}
Higham, N.~J. and Strabi{\'c}, N.
\newblock Anderson acceleration of the alternating projections method for
  computing the nearest correlation matrix.
\newblock \emph{Numerical Algorithms}, 72\penalty0 (4):\penalty0 1021--1042,
  2016.

\bibitem[Hoffmann(1992)]{hoffmann1992distance}
Hoffmann, A.
\newblock The distance to the intersection of two convex sets expressed by the
  distances to each of them.
\newblock \emph{Mathematische Nachrichten}, 157\penalty0 (1):\penalty0 81--98,
  1992.

\bibitem[Kavis et~al.(2019)Kavis, Levy, Bach, and Cevher]{kavis2019unixgrad}
Kavis, A., Levy, K.~Y., Bach, F., and Cevher, V.
\newblock Unixgrad: A universal, adaptive algorithm with optimal guarantees for
  constrained optimization.
\newblock In \emph{Proceedings of the 33rd International Conference on Neural
  Information Processing Systems}, 2019.

\bibitem[Kundu et~al.(2018)Kundu, Bach, and Bhattacharya]{kundu2018convex}
Kundu, A., Bach, F., and Bhattacharya, C.
\newblock Convex optimization over intersection of simple sets: improved
  convergence rate guarantees via an exact penalty approach.
\newblock In \emph{International Conference on Artificial Intelligence and
  Statistics}, pp.\  958--967. PMLR, 2018.

\bibitem[LeCun(1998)]{lecun1998mnist}
LeCun, Y.
\newblock The mnist database of handwritten digits.
\newblock \emph{http://yann. lecun. com/exdb/mnist/}, 1998.

\bibitem[Levy(2017)]{levy2017online}
Levy, K.~Y.
\newblock Online to offline conversions, universality and adaptive minibatch
  sizes.
\newblock In \emph{Proceedings of the 31st International Conference on Neural
  Information Processing Systems}, 2017.

\bibitem[Levy et~al.(2018)Levy, Yurtsever, and Cevher]{levy2018online}
Levy, K.~Y., Yurtsever, A., and Cevher, V.
\newblock Online adaptive methods, universality and acceleration.
\newblock In \emph{Proceedings of the 32nd International Conference on Neural
  Information Processing Systems}, 2018.

\bibitem[Lewis et~al.(2004)Lewis, Yang, Russell-Rose, and Li]{lewis2004rcv1}
Lewis, D.~D., Yang, Y., Russell-Rose, T., and Li, F.
\newblock Rcv1: A new benchmark collection for text categorization research.
\newblock \emph{Journal of machine learning research}, 5\penalty0
  (Apr):\penalty0 361--397, 2004.

\bibitem[Malitsky \& Pock(2018)Malitsky and Pock]{malitsky2018first}
Malitsky, Y. and Pock, T.
\newblock A first-order primal-dual algorithm with linesearch.
\newblock \emph{SIAM Journal on Optimization}, 28\penalty0 (1):\penalty0
  411--432, 2018.

\bibitem[Mishchenko \& Richt{\'a}rik(2019)Mishchenko and
  Richt{\'a}rik]{mishchenko2019stochastic}
Mishchenko, K. and Richt{\'a}rik, P.
\newblock A stochastic decoupling method for minimizing the sum of smooth and
  non-smooth functions.
\newblock \emph{arXiv preprint arXiv:1905.11535}, 2019.

\bibitem[Nesterov(2003)]{nesterov2003introductory}
Nesterov, Y.
\newblock \emph{Introductory lectures on convex optimization: A basic course},
  volume~87.
\newblock Springer Science \& Business Media, 2003.

\bibitem[Nesterov(2015)]{nesterov2015universal}
Nesterov, Y.
\newblock Universal gradient methods for convex optimization problems.
\newblock \emph{Mathematical Programming}, 152\penalty0 (1):\penalty0 381--404,
  2015.

\bibitem[Ouyang et~al.(2013)Ouyang, He, Tran, and Gray]{ouyang2013stochastic}
Ouyang, H., He, N., Tran, L., and Gray, A.
\newblock Stochastic alternating direction method of multipliers.
\newblock In \emph{International Conference on Machine Learning}, pp.\  80--88.
  PMLR, 2013.

\bibitem[Pedregosa(2016)]{pedregosa2016convergence}
Pedregosa, F.
\newblock On the convergence rate of the three operator splitting scheme.
\newblock \emph{arXiv preprint arXiv:1610.07830}, 2016.

\bibitem[Pedregosa \& Gidel(2018)Pedregosa and Gidel]{pedregosa2018adaptive}
Pedregosa, F. and Gidel, G.
\newblock Adaptive three operator splitting.
\newblock In \emph{International Conference on Machine Learning}, pp.\
  4085--4094, 2018.

\bibitem[Pedregosa et~al.(2019)Pedregosa, Fatras, and
  Casotto]{pedregosa2019proximal}
Pedregosa, F., Fatras, K., and Casotto, M.
\newblock Proximal splitting meets variance reduction.
\newblock In \emph{The 22nd International Conference on Artificial Intelligence
  and Statistics}, pp.\  1--10. PMLR, 2019.

\bibitem[Pedregosa et~al.(2020)Pedregosa, Negiar, and Dresdner]{copt}
Pedregosa, F., Negiar, G., and Dresdner, G.
\newblock copt: composite optimization in python.
\newblock 2020.
\newblock \doi{10.5281/zenodo.1283339}.
\newblock URL \url{http://openo.pt/copt/}.

\bibitem[Raguet et~al.(2013)Raguet, Fadili, and
  Peyr{\'e}]{raguet2013generalized}
Raguet, H., Fadili, J., and Peyr{\'e}, G.
\newblock A generalized forward-backward splitting.
\newblock \emph{SIAM Journal on Imaging Sciences}, 6\penalty0 (3):\penalty0
  1199--1226, 2013.

\bibitem[Rakhlin \& Sridharan(2013)Rakhlin and
  Sridharan]{rakhlin2013optimization}
Rakhlin, A. and Sridharan, K.
\newblock Optimization, learning, and games with predictable sequences.
\newblock In \emph{Proceedings of the 26th International Conference on Neural
  Information Processing Systems-Volume 2}, pp.\  3066--3074, 2013.

\bibitem[Salim et~al.(2020)Salim, Condat, Mishchenko, and
  Richt{\'a}rik]{salim2020dualize}
Salim, A., Condat, L., Mishchenko, K., and Richt{\'a}rik, P.
\newblock Dualize, split, randomize: Fast nonsmooth optimization algorithms.
\newblock \emph{arXiv preprint arXiv:2004.02635}, 2020.

\bibitem[Scardapane et~al.(2017)Scardapane, Comminiello, Hussain, and
  Uncini]{scardapane2017group}
Scardapane, S., Comminiello, D., Hussain, A., and Uncini, A.
\newblock Group sparse regularization for deep neural networks.
\newblock \emph{Neurocomputing}, 241:\penalty0 81--89, 2017.

\bibitem[Shivanna et~al.(2015)Shivanna, Chatterjee, Sankaran, Bhattacharyya,
  and Bach]{shivanna2015spectral}
Shivanna, R., Chatterjee, B., Sankaran, R., Bhattacharyya, C., and Bach, F.
\newblock Spectral norm regularization of orthonormal representations for graph
  transduction.
\newblock In \emph{Neural Information Processing Systems}, 2015.

\bibitem[Tibshirani et~al.(2011)Tibshirani, Hoefling, and
  Tibshirani]{tibshirani2011nearly}
Tibshirani, R.~J., Hoefling, H., and Tibshirani, R.
\newblock Nearly-isotonic regression.
\newblock \emph{Technometrics}, 53\penalty0 (1):\penalty0 54--61, 2011.

\bibitem[Toh et~al.(1999)Toh, Todd, and T{\"u}t{\"u}nc{\"u}]{toh1999sdpt3}
Toh, K.-C., Todd, M.~J., and T{\"u}t{\"u}nc{\"u}, R.~H.
\newblock {SDPT3}—a {MATLAB} software package for semidefinite programming,
  version 1.3.
\newblock \emph{Optimization methods and software}, 11\penalty0 (1-4):\penalty0
  545--581, 1999.

\bibitem[V{\~u}(2013)]{vu2013splitting}
V{\~u}, B.~C.
\newblock A splitting algorithm for dual monotone inclusions involving
  cocoercive operators.
\newblock \emph{Advances in Computational Mathematics}, 38\penalty0
  (3):\penalty0 667--681, 2013.

\bibitem[Yan(2018)]{yan2018new}
Yan, M.
\newblock A new primal--dual algorithm for minimizing the sum of three
  functions with a linear operator.
\newblock \emph{Journal of Scientific Computing}, 76\penalty0 (3):\penalty0
  1698--1717, 2018.

\bibitem[Yuan et~al.(2011)Yuan, Liu, and Ye]{yuan2011efficient}
Yuan, L., Liu, J., and Ye, J.
\newblock Efficient methods for overlapping group lasso.
\newblock \emph{Advances in neural information processing systems},
  24:\penalty0 352--360, 2011.

\bibitem[Yuan \& Lin(2006)Yuan and Lin]{yuan2006model}
Yuan, M. and Lin, Y.
\newblock Model selection and estimation in regression with grouped variables.
\newblock \emph{Journal of the Royal Statistical Society: Series B (Statistical
  Methodology)}, 68\penalty0 (1):\penalty0 49--67, 2006.

\bibitem[Yurtsever et~al.(2016)Yurtsever, V{\~u}, and
  Cevher]{yurtsever2016stochastic}
Yurtsever, A., V{\~u}, B.~C., and Cevher, V.
\newblock Stochastic three-composite convex minimization.
\newblock In \emph{Proceedings of the 30th International Conference on Neural
  Information Processing Systems}, pp.\  4329--4337, 2016.

\bibitem[Yurtsever et~al.(2018)Yurtsever, Fercoq, Locatello, and
  Cevher]{yurtsever2018conditional}
Yurtsever, A., Fercoq, O., Locatello, F., and Cevher, V.
\newblock A conditional gradient framework for composite convex minimization
  with applications to semidefinite programming.
\newblock In \emph{International Conference on Machine Learning}, pp.\
  5727--5736. PMLR, 2018.

\bibitem[Yurtsever et~al.(2021)Yurtsever, Mangalick, and
  Sra]{yurtsever2021three}
Yurtsever, A., Mangalick, V., and Sra, S.
\newblock Three operator splitting with a nonconvex loss function.
\newblock In \emph{International Conference on Machine Learning}, pp.\
  12267--12277. PMLR, 2021.

\bibitem[Zeng \& So(2018)Zeng and So]{Zeng2018}
Zeng, W.-J. and So, H.~C.
\newblock Outlier--robust matrix completion via $\ell_p$-minimization.
\newblock \emph{{IEEE} Trans. on Sig. Process}, 66\penalty0 (5):\penalty0
  1125--1140, 2018.

\bibitem[Zhao \& Cevher(2018)Zhao and Cevher]{zhao2018stochastic}
Zhao, R. and Cevher, V.
\newblock Stochastic three-composite convex minimization with a linear
  operator.
\newblock In \emph{International Conference on Artificial Intelligence and
  Statistics}, pp.\  765--774. PMLR, 2018.

\bibitem[Zhao et~al.(2019)Zhao, Haskell, and Tan]{zhao2019optimal}
Zhao, R., Haskell, W.~B., and Tan, V.~Y.
\newblock An optimal algorithm for stochastic three-composite optimization.
\newblock In \emph{The 22nd International Conference on Artificial Intelligence
  and Statistics}, pp.\  428--437. PMLR, 2019.

\end{thebibliography}
\bibliographystyle{icml2021} 

\end{document}